\documentclass[12pt,a4paper]{article}

\usepackage[left=25mm,right=25mm,top=25mm,bottom=25mm]{geometry}
\usepackage{fontspec}
\usepackage{amsmath,amssymb,amsthm,mathtools}
\usepackage{tikz}
\usepackage{caption}
\usepackage{float}
\usepackage{cite}
\usepackage[colorlinks=true,linkcolor=blue,citecolor=blue,urlcolor=blue]{hyperref}

\newtheorem{theorem}{Theorem}[section]
\newtheorem{conjecture}[theorem]{Conjecture}
\newtheorem{lemma}[theorem]{Lemma}
\newtheorem{corollary}[theorem]{Corollary}
\newtheorem{proposition}[theorem]{Proposition}

\newcommand{\comp}{\overline}
\newcommand{\eps}{\varepsilon}
\newcommand{\resultcite}[1]{[\citen{#1}]}

\title{The maximum size of simple solid\\matching covered graphs}
\author{Tong Zhang, Wei Li\thanks{ Corresponding author.
		E-mail addresses: liw@nwpu.edu.cn.}\\[0.6em]
\normalsize School of Mathematics and Statistics,\\[0.35em]
\normalsize Northwestern Polytechnical University, Xi'an, Shaanxi, 710129,
P.~R.~China.}
\date{}

\begin{document}

\maketitle

\begin{abstract}
A connected graph with at least two vertices is matching covered if each of
its edges is contained in a perfect matching.
 A matching covered graph is solid if every separating cut in it is a tight cut. A matching covered graph which is free of nontrivial tight cuts is a brick if it is nonbipartite.
Every bipartite matching covered graph is solid. Lucchesi and Murty conjectured that there exists a positive integer $N$ such that, for every integer $n\ge N$, the maximum number of edges in a simple solid matching covered graph on $2n$ vertices is $n^2$. In this paper, we disprove this Conjecture, and show that the maximum number of edges of a simple solid matching covered graph on $2n$ ($n\ge2$) vertices is $n^2+2$. Moreover, we characterize the graphs attaining this bound. In addition, we prove that every simple solid brick of order $2n$ has at most $n^2$ edges for $n\ge4$.
\end{abstract}


\section{Introduction}
% The cut $C$ is tight if %$|M\cap C|=1$ for every perfect matching $M$ of $G$.
%A cut $C$ of a matching covered graph $G$ is a separating cut if both $C$-%contractions of $G$ are also matching covered.

All graphs considered in this paper are loopless, but they may have
multiple edges. A graph is \emph{simple} if it has no multiple edges. For
a graph $G$, let $V(G)$ and $E(G)$ denote its vertex set and edge set,
respectively. The \emph{order} of $G$ is $|V(G)|$, and the \emph{size} of $G$ is
$|E(G)|$. For a vertex $v$ in $V(G)$, the neighbour set of $v$ is the set of all
vertices of $G$ adjacent to $v$, denoted by $N_G(v)$. The degree of $v$ is
denoted by $d(v)$.
% When $G$ is simple, it is the size $|N_G(v)|$ of the neighbour
% set; in general, it is the number of edges incident with $v$.
Let $\delta(G)$ and $\Delta(G)$ denote the minimum and maximum degrees of $G$, respectively.
A set $S\subseteq V(G)$ is a
\emph{stable set} if no edge of
$G$ has both end vertices in $S$. For $X\subseteq V(G)$, the subgraph of $G$
induced by $X$, denoted by $G[X]$, has vertex set $X$ and edge set consisting
of all edges of $G$ with both end vertices in $X$. The complement of a simple
graph $G$is denoted by $\comp G$.
%has vertex set $V(G)$, and two distinct
%vertices are adjacent in $\comp G$ if and only if they are not adjacent in
%$G$.
The union of two vertex-disjoint graphs $G$
and $H$ is denoted by $G\cup H$.
%is the graph with vertex set
%$V(G)\cup V(H)$ and edge set $E(G)\cup E(H)$.
The join of $G$ and $H$, denoted by $G\vee H$, is obtained from $G\cup H$ by joining each vertex of
$G$ to every vertex of $H$.
%A graph is \emph{bipartite} if its vertex set can
%be partitioned into two stable sets $A$ and $B$. A graph is \emph{connected} if
%every pair of its vertices is joined by a path.
A \emph{component} of a graph
$G$ is a maximal connected subgraph of $G$. A component is \emph{odd} or
\emph{even} according as its order is odd or even.
%It is \emph{trivial} if it
%has order one, and is \emph{nontrivial} otherwise.
A \emph{vertex cut} of a graph
$G$ is a subset $S\subseteq V(G)$ such that $G-S$ is disconnected. A
\emph{$k$-vertex-cut} is a vertex cut of cardinality $k$.
%For a positive integer
%$k$, a graph is \emph{$k$-connected} if it has at least $k+1$ vertices and has
%no vertex cut of cardinality less than $k$.
A \emph{matching} in $G$ is a set $M\subseteq E(G)$ such that no vertex of $G$
is incident with more than one edge of $M$. It is \emph{perfect} if every
vertex is incident with precisely one edge of $M$. A graph is \emph{matchable}
if it has a perfect matching. A \emph{matching covered} graph is a connected
graph on at least two vertices in which each edge is covered by a perfect
matching.

For disjoint sets $X,Y\subseteq V(G)$, let
$E[X,Y]$ denote the set of edges with one end vertex in $X$ and the other in $Y$. For $X\subseteq V(G)$, set
$\comp X=V(G)-X$. If $X$ is nonempty, let
$\partial(X)=E[X,\comp X]$, called the \emph{edge cut}
\emph{cut} associated with $X$. If $G$ is connected, the sets $X$ and
$\comp X$ are the \emph{shores} of $\partial(X)$. A cut $\partial(X)$ is odd
if both of its shores have odd cardinality. The cut $\partial(X)$ is
\emph{trivial} if either $|X|=1$ or $|\comp X|=1$. The two $\partial(X)$-contractions
of $G$, denoted by $G/X$ and $G/\comp X$, are obtained by contracting $X$ and
$\comp X$, respectively, to single new vertices, removing all resulting loops
and retaining all resulting parallel edges.
A cut $C$ of a matching covered graph $G$ is a \emph{separating cut} if
both $C$-contractions of $G$ are matching covered. The cut $C$ is a
\emph{tight cut} if $|M\cap C|=1$ for every perfect matching $M$ of $G$.
Every tight cut of a matching covered graph is a separating cut~\cite{lovasz-lattice}. A matching
covered graph which is free of nontrivial tight cuts is a \emph{brick}
if it is nonbipartite and a \emph{brace} if it is bipartite. Lov\'asz showed that
any two tight cut decompositions of a matching covered graph yield the same
list of bricks and braces, up to multiple edges~\cite{lovasz-lattice}.
%A brick $G$ is
%\emph{minimal} if $G-e$ is not a brick for every edge $e\in E(G)$. A brick $G$
%is \emph{extremal} if it has precisely $|E(G)|-|V(G)|+1$ perfect matchings.
A matching covered graph is \emph{solid} if every separating cut is a tight cut.
In particular, a brick is solid if it is free of nontrivial separating cuts.

The investigations on solid matching covering graphs began in the 2000s. Carvalho et al. discovered solid matching covered graphs in
connection with their work on a conjecture of Lov\'asz, which asserts that every
brick $G$ different from $K_4$, the triangular prism and the Petersen graph has an
edge $e$ such that $G-e$ is matching covered and has exactly one brick in a tight
cut decomposition~\cite{clm-characteristic}. They then found that solid matching covered graphs have many properties akin to those of bipartite
matching covered graphs~\cite{independent-perfect-matchings}.
One such similarity is that the perfect matching polytopes of the two classes
admit analogous descriptions~\cite{polytope-solid}. Carvalho et al.
showed that a matching covered graph is solid if and only if
each of its bricks is solid~\cite{clm-characteristic}. Thus, recognizing solid matching
covered graphs reduces to recognizing solid bricks. Following this, a plenty of research have been done on solid and nonsolid bricks ~\cite{cubic-solid,extremal-nonsolid}.
%Chen, Feng, Lu and Zhang described an infinite family of cubic solid
%bricks containing two vertex-disjoint odd cycles~\cite{cubic-solid}. Feng and Lu %presented an infinite family of
%counterexamples to a conjecture of Carvalho, Lucchesi and Murty which asserted
%that, apart from the Petersen graph, every extremal nonsolid brick is obtained
%from an extremal brick by replacing a cubic vertex with a triangle, up to multiple
%edges~\cite{extremal-nonsolid}.
A natural problem is to determine the maximum number of edges of a simple solid matching covered graph with given order.
Lucchesi and Murty found that the maximum size of a simple solid matching
covered graph on four vertices is $6$, and the corresponding maximum
on six vertices is $11$~\cite{lm-book}. Moreover,
a simple bipartite graph of order $2n$ has at most $n^2$ edges, and
$|E(K_{n,n})|=n^2$. Since every bipartite matching covered graph is solid,
the maximum number of edges of a simple solid matching covered graph on $2n$ vertices
is at least $n^2$. Thus, they proposed the following conjecture.

\begin{conjecture}~\cite{lm-book}\label{conj:LM}
There exists a positive integer $N$ such that, for every integer $n\ge N$,
every simple solid matching covered graph $G$ of order $2n$ satisfies
$|E(G)|\le n^2$.
\end{conjecture}

The study of bricks has already produced several results concerning their number of edges and the number of cubic vertices.
In 2006, Norine and Thomas proved that every minimal brick of
order $2n$, for $n\ge5$, has at most $5n-7$ edges~\cite{minimal-bricks-nt}.
 Nearly two decades later, Feng and Yan characterized all minimal
bricks attaining equality~\cite{minimal-bricks}.  Lin, Zhang and Lu proved that every minimal brick has at least four cubic
vertices~\cite{cubic-vertices-minimal}. He and Lu proved that
every solid minimal brick $G$ has at least $\frac{2}{5}|V(G)|$ cubic
vertices~\cite{solid-minimal-cubic}. For minimal bricks with average degree at least
$\frac{23}{5}$, Feng verified the Norine--Thomas conjecture, which asserts that
there exists a positive constant $\alpha$ such that every minimal brick $G$ has at
least $\alpha|V(G)|$ cubic vertices~\cite{norine-thomas-feng}. These results all
concern minimal bricks. A
solid brick need not be minimal. By now,  the maximum number of
edges in a solid brick has not been determined. In this paper, we first focus on
this problem, and find the bound for the size of simple solid bricks, which verifies that Conjecture 1 holds for simple solid bricks.

\begin{theorem}\label{thm:brick-bound}
Let $G$ be a simple solid brick of order $2n$ for $n\ge4$. Then
\[
 |E(G)|\le n^2.
\]
\end{theorem}

For the size of simple solid matching covered graphs, we give counter examples to disprove Conjecture 1. Moreover,
we prove that for every integer $n\ge2$,  each
simple solid matching covered graph of order $2n$ has at most $n^2+2$ edges,
and characterize the graphs attaining the upper bound.
Our result is as below.

\begin{theorem}\label{thm:main}
Let $G$ be a simple solid matching covered graph of order $2n$, for $n\ge2$.
Then
\[
 |E(G)|\le n^2+2.
\]
Equality holds if and only if one of the following statements holds
\begin{enumerate}
\item[(i)] if $n=2$, $G\cong A_2$;
\item[(ii)] if $n=3$, $G\cong A_3$ or $G\cong D_3$;
\item[(iii)] if $n\ge4$, $G\cong A_n$, $G\cong B_n$, or $G\cong D_n$.
\end{enumerate}
\end{theorem}

The rest of the paper is organized as follows. In Section~2, we present some definitions and results used in this paper. In Section~3, we prove Theorem~\ref{thm:brick-bound} to show that Conjecture \ref{conj:LM} is true for simple solid brick. In Section~4, we first construct the graphs $A_n$, $B_n$ and $D_n$ to disprove Conjecture 1. Then we mainly prove Theorem~\ref{thm:main}
to give a tight upper bounds on the maximum number of edges of simple solid matching covered graphs.

%show that the maximum size of a simple solid matching covered graph is $n^2+2$, and %determine the graphs achieving the upper bound.

%%%%%%%%%%%%%%%%%%%%%%%%%%%%%%%%%%%%%%%%%%%%%

\section{Preliminaries}

This section collects some definitions and results used in our proofs. For a graph $G$, let
$\operatorname{Sim}(G)$ denote its underlying simple graph, obtained by replacing
each set of parallel edges by a single edge.
A graph $G$ is \emph{factor-critical} if $G-v$ is matchable for every vertex
$v$ of $G$. A nontrivial graph $G$ is \emph{bicritical} if $G-u-v$ is matchable
for any two distinct vertices $u$ and $v$ of $G$.

In 1982, Edmonds et al. discovered the following characterization
of bricks in terms of bicritical graphs.

\begin{theorem}[\resultcite{elp}]\label{thm:brick-characterization}
A nonbipartite matching covered graph is a brick if and only if it is
3-connected and bicritical.
\end{theorem}

The following characterization of solid bricks will be used throughout
Section~3.

\begin{theorem}[\resultcite{polytope-solid}]\label{thm:odd-bicycle}
A brick $G$ is nonsolid if and only if it has two disjoint odd cycles $C_1$
and $C_2$ such that $G-(V(C_1)\cup V(C_2))$ has a perfect
matching.
\end{theorem}

Here we introduce some basic properties on solid matching covered graphs.

\begin{lemma}[\resultcite{plummer-extendable}]\label{lem:two-connected}
Every matching covered graph of order at least four is 2-connected.
\end{lemma}

\begin{theorem}[\resultcite{clm-characteristic}]\label{thm:bipartite-solid}
Every bipartite matching covered graph is solid.
\end{theorem}

\begin{lemma}\label{lem:simplification}
Let $G$ be a matching covered graph, and let $X$ be a nonempty proper subset
of $V(G)$. Then the following statements hold.
\begin{enumerate}
\item[(i)] The underlying simple graph $\operatorname{Sim}(G)$ is matching
covered.
\item[(ii)] $\partial_G(X)$ is a tight cut if and only if
$\partial_{\operatorname{Sim}(G)}(X)$ is a tight cut.
\item[(iii)] $\partial_G(X)$ is a separating cut if and only if
$\partial_{\operatorname{Sim}(G)}(X)$ is a separating cut. Consequently, $G$ is
solid if and only if its underlying simple graph is solid.
\end{enumerate}
\end{lemma}

\begin{proof}
For (i), observe first that a graph is connected if and only if its underlying
simple graph is connected. Let $e'$ be an edge of $\operatorname{Sim}(G)$ with end vertices $u$
and $v$, and choose an edge $e$ of $G$ with the same end vertices. Since $G$ is
matching covered, $e$ lies in a perfect matching $M$ of $G$. For each edge
$f\in M$,
let $f'$ be the unique edge of $\operatorname{Sim}(G)$ with the same end vertices
as $f$, and set $M'=\{f':f\in M\}$. Then $M'$ is a perfect matching of
$\operatorname{Sim}(G)$ containing $e'$. Hence $\operatorname{Sim}(G)$ is
matching covered. Moreover, let $H$ be a graph such that
$\operatorname{Sim}(H)$ is matching covered, and let $e$ be any edge of $H$.
Let $e'$ be the edge of $\operatorname{Sim}(H)$ with the same end vertices as
$e$, and let $N'$ be a perfect matching containing $e'$. Replacing $e'$ by $e$
and each other edge of $N'$ by an edge of $H$ with the same end vertices yields
a perfect matching of $H$ containing $e$. Thus, for any graph $H$, the graph
$H$ is matching covered if and only if its underlying simple graph is matching
covered.

We now prove (ii). Let $M$ be a perfect matching of $G$, and let $M'$ be the
set of edges of $\operatorname{Sim}(G)$ having the same end vertices as the
edges of $M$. Then $M'$ is a perfect matching of $\operatorname{Sim}(G)$ and
\[
 |M\cap\partial_G(X)|
 =|M'\cap\partial_{\operatorname{Sim}(G)}(X)|.
\]
Conversely, for every perfect matching $N'$ of $\operatorname{Sim}(G)$, choose
in $G$ one edge with the same end vertices as each edge of $N'$, and let $N$ be
the set of chosen edges. Then $N$ is a perfect matching of $G$ and
\[
 |N\cap\partial_G(X)|
 =|N'\cap\partial_{\operatorname{Sim}(G)}(X)|.
\]
It follows that $\partial_G(X)$ is a tight cut if and only if
$\partial_{\operatorname{Sim}(G)}(X)$ is a tight cut.

For each $S\in\{X,\comp X\}$, the graphs $G/S$ and
$\operatorname{Sim}(G)/S$ have isomorphic underlying simple graphs. Hence
$G/S$ is matching covered if and only if $\operatorname{Sim}(G)/S$ is matching
covered.
It follows that $\partial_G(X)$ is a separating cut if and only if
$\partial_{\operatorname{Sim}(G)}(X)$ is a separating cut. Therefore, $G$ is solid if and only if
$\operatorname{Sim}(G)$ is solid. This proves the result in (iii).
\end{proof}

%%%%%%%%%%%%%%%%%%%%%%%%%%%%%%%%%%%%%%%%%%%%%%%%%%%%%%%%%%%%%%%%%%%%%%%%%
\section{The maximum size of simple solid bricks}
%%%%%%%%%%%%%%%%%%%%%%%%%%%%%%%%%%%%%%%%%%%%%%%%%%%%%%%%%%%%%%%%%%%%%%%%%

In this section, we first present some results on the vertex degree of simple solid bricks. Using these results, we prove
Theorem~\ref{thm:brick-bound} to show that the maximum size of simple solid bricks of order $2n$ is $n^2$. This means that Conjecture \ref{conj:LM} is true for simple solid bricks.

Let $G$ be a simple solid brick of order $2n$ and let $x$ and $y$
two distinct vertices of $G$. By
Theorem~\ref{thm:brick-characterization}, the graph $G-x-y$ has a perfect
matching $M$. We define the symbol $\eps$ as
below
\begin{equation}\label{eq:epsilon}
\eps=
\begin{cases}
1, & \text{if } xy\in E(G),\\
0, & \text{otherwise}.
\end{cases}
\end{equation}
For $v\in\{x,y\}$, define $A_v=N_G(v)\cap V(G-x-y)$. Throughout this section, the symbols $s_v$, $r_v$ and
$z_v$ denote, respectively, the numbers of edges of $M$ having both end
vertices in $A_v$, neither end vertex in $A_v$ and exactly one end vertex in
$A_v$.

Now we show an identity on the vertex degree of simple solid bricks.

\begin{lemma}\label{lem:identity}
Let $G$ be a simple solid brick of order $2n$, and let $x$ and $y$ be two distinct
vertices of $G$. Then, for
each $v\in\{x,y\}$,
\[
 s_v-r_v=d(v)-\eps-(n-1),
\]
where $\eps$ is given by \eqref{eq:epsilon}.
\end{lemma}

\begin{proof}
Since $G$ is a matching covered graph of order $2n$, the graph $G-x-y$ has
order $2n-2$. Thus the perfect matching $M$ of $G-x-y$ has $n-1$ edges, and
hence
\[
 s_v+z_v+r_v=n-1.
\]
Since $M$ covers all vertices in $A_v$, a simple counting argument shows that
\[
 2s_v+z_v=|A_v|=d(v)-\eps.
\]
Subtracting $s_v+z_v+r_v=n-1$ from $2s_v+z_v=d(v)-\eps$ yields the asserted
equality.
\end{proof}

\begin{lemma}\label{lem:common-edge}
Let $G$ be a simple solid brick of order $2n$, let $x$ and $y$ be two distinct
vertices of $G$, and let $M$ be a perfect matching of $G-x-y$. If $s_x>0$
and $s_y>0$, then $s_x=s_y=1$, and the unique edge of $M$ with
both end vertices in $A_x$ is also the unique edge of $M$ with both end
vertices in $A_y$.
\end{lemma}

\begin{proof}
Let $ab$ be an edge of $M$ with both end vertices in $A_x$, and let $pq$ be an
edge of $M$ with both end vertices in $A_y$. If $ab\ne pq$, let $C_1=xabx$ and
$C_2=ypqy$. These are
disjoint triangles, and $M-\{ab,pq\}$ is a perfect matching of
$G-(V(C_1)\cup V(C_2))$. By
Theorem~\ref{thm:odd-bicycle}, $G$ is a nonsolid brick, a contradiction. Thus every edge of $M$ with both end vertices
in $A_x$ and every edge of $M$ with both end vertices in $A_y$ must be the
same edge. Hence
$s_x=s_y=1$.
\end{proof}

\begin{corollary}\label{cor:degree-pairs}
Let $G$ be a simple solid brick of order $2n$, let $x$ and $y$ be two distinct
vertices of $G$, and let $M$ be a perfect matching of $G-x-y$. Then the
following statements hold.
\begin{enumerate}
\item[(i)] If $xy\notin E(G)$ and $d(x),d(y)\ge n$, then
$d(x)=d(y)=n$.
\item[(ii)] If $xy\in E(G)$ and $d(x),d(y)\ge n+1$, then
$d(x)=d(y)=n+1$.
\end{enumerate}
\end{corollary}

\begin{proof}
It follows from Lemma~\ref{lem:identity} and the hypotheses of (i) that
\[
 s_x-r_x=d(x)-(n-1)\ge1,
\]
and
\[
 s_y-r_y=d(y)-(n-1)\ge1.
\]
Thus $s_x,s_y>0$. By Lemma~\ref{lem:common-edge}, $s_x=s_y=1$, and hence
\[
 d(x)-(n-1)=1-r_x\le1,
\]
and
\[
 d(y)-(n-1)=1-r_y\le1.
\]
Combining the above inequalities, we obtain that
$d(x)=d(y)=n$. This proves part (i).

Using Lemma~\ref{lem:identity} together with the hypotheses of (ii), we have
\[
 s_x-r_x=d(x)-n\ge1,
\]
and
\[
 s_y-r_y=d(y)-n\ge1.
\]
Consequently, $s_x,s_y>0$. Applying Lemma~\ref{lem:common-edge} gives
$s_x=s_y=1$, and hence
\[
 d(x)-n=1-r_x\le1,
\]
and
\[
 d(y)-n=1-r_y\le1.
\]
Combining the above inequalities, we get that
$d(x)=d(y)=n+1$. This proves part (ii).
\end{proof}

We are now ready to prove Theorem~\ref{thm:brick-bound}.

\begin{proof}[\textbf{\upshape Proof of Theorem~\ref{thm:brick-bound}}]
Let $m=|E(G)|$ and suppose to the contrary that $m>n^2$. It follows that
\begin{equation}
 2m\ge2n^2+2. \label{eq:brick-lower}
\end{equation}
If $\Delta(G)\le n$, then $2m\le2n^2$, contrary to
\eqref{eq:brick-lower}. Hence $\Delta(G)\ge n+1$. In the following, we
consider two cases, depending on the maximum degree of $G$.

\smallskip
\noindent\textbf{\upshape Case 1: $\Delta(G)=n+1$.}
For $v\in V(G)$, let
\[
 U=\{v:d(v)=n+1\},\qquad
 R=\{v:d(v)=n\},\qquad
 L=V(G)-(U\cup R),
\]
and let $k=|U|$, $r=|R|$ and $\ell=|L|$. By
Corollary~\ref{cor:degree-pairs}(i), any two distinct vertices of $U$ are
adjacent; if not, they both would have degree $n$, contrary to the definition
of $U$. Thus the subgraph $G[U]$ induced by $U$ is a complete graph. Similarly,
every vertex of $U$ is adjacent
to every vertex of $R$. Consequently, each vertex of $U$ has $k-1+r$
neighbours in $U\cup R$. Since the degree of each vertex in $U$ is $n+1$, we have
$k-1+r\le n+1$, and hence
\begin{equation}
 k+r\le n+2. \label{eq:kr}
\end{equation}
Set $t=n+2-k-r$. By \eqref{eq:kr}, $t$ is a nonnegative integer, and each
vertex of $U$ has exactly $t$ neighbours in $L$.
Every vertex of $U$ has degree $n+1$, every vertex of $R$ has degree
$n$, and every vertex of $L$ has degree at most $n-1$. Then it follows from
\eqref{eq:brick-lower} that
\begin{equation}
 2n^2+2\le2m\le k(n+1)+rn+\ell(n-1)=2n^2+k-\ell.
 \label{eq:degree-sum}
\end{equation}
It leads to that
\begin{equation}
 k-\ell\ge2. \label{eq:kell}
\end{equation}
Since $U$, $R$ and $L$ partition $V(G)$, we have $k+r+\ell=2n$. Hence
$\ell=2n-k-r=n-2+(n+2-k-r)=n-2+t$ and $k=n+2-t-r$. It follows from these two
equalities and \eqref{eq:kell} that $2t+r\le2$. Since $t$ and $r$ are nonnegative integers,
\(
 (t,r)\in\{(0,2),(0,1),(0,0),(1,0)\}.
\)
Following this, the corresponding triples $(k,r,\ell)$ are
\begin{equation}
 (k,r,\ell)\in\{(n,2,n-2),(n+1,1,n-2),(n+2,0,n-2),(n+1,0,n-1)\}.
 \label{eq:triples}
\end{equation}

For the first three triples in \eqref{eq:triples}, $k+r=n+2$, and
$t=0$. Thus no edge joins a vertex of $U$ to a vertex of $L$. Since $n\ge4$,
the sets $U$ and $L$ are nonempty. Thus $G-R$ is
disconnected. As $|R|\le2$, this contradicts the fact that $G$ is 3-connected.

For the fourth triple $(k,r,\ell)=(n+1,0,n-1)$ in \eqref{eq:triples}, it follows from
\eqref{eq:degree-sum} that $2n^2+2\le2m\le2n^2+k-\ell=2n^2+2$. Hence equality
holds throughout in \eqref{eq:degree-sum}.
Together with the facts that $G[U]$ is a complete graph and
$k=n+1$, this yields
$m=n^2+1$, $G[U]\cong K_{n+1}$, and $d(v)=n-1$ for $v\in L$.
Since every vertex of $U$ has exactly $t$ neighbours in $L$ and
$t=1$, we have $|E[U,L]|=|U|=n+1$. Then,
\[
 \sum_{v\in L} d(v)=(n-1)^2=2|E(G[L])|+(n+1).
\]
It follows that $|E(G[L])|=\frac{n(n-3)}{2}=\frac{(n-1)(n-2)}{2}-1$.
Hence $G[L]$ is obtained from $K_{n-1}$ by
deleting one edge. Let $a$ and $b$ denote the two nonadjacent vertices of
$G[L]$, and let $G[L]+ab$ denote the graph obtained from $G[L]$ by adding the
edge $ab$. Then
\begin{equation}
 G[L]+ab\cong K_{n-1}. \label{eq:L-structure}
\end{equation}
Since $d(v)=n-1$ for $v\in L$ and $R$ is empty, the vertices $a$
and $b$ each have two neighbours in $U$, and every vertex of $L-\{a,b\}$ has
one neighbour in $U$.  Now, suppose that
$N_G(a)\cap U=\{u_1,u_2\}$ and
$N_G(b)\cap U=\{v_1,v_2\}$.
We consider the following three cases separately.

Suppose first that $n=4$. Then $(k,r,\ell)=(5,0,3)$, and $t=1$.
Since $\ell=3$, there is a third vertex $c$ in $L$ distinct from $a$ and $b$.
Since every vertex of $U$ has exactly one neighbour in $L$, no two distinct
vertices of $L$ have a common neighbour in $U$. Let $C_1=au_1u_2a$ and
$C_2=bv_1v_2b$. Then they are disjoint triangles. The graph
$G-(V(C_1)\cup V(C_2))$ has exactly two vertices, $c$ and $w$, where $w$ is the
unique neighbour of $c$ in $U$. Hence the edge $cw$ is a perfect matching of
$G-(V(C_1)\cup V(C_2))$.
By Theorem~\ref{thm:odd-bicycle}, this contradicts the hypothesis that
$G$ is a solid brick.

Suppose next that $n\ge6$ is even. Then $|U|\ge5$ and $|L|\ge5$. Let $c$ and $d$ be two distinct vertices
of $L-\{a,b\}$. Since $G[L]$ is obtained from $K_{n-1}$ by deleting $ab$,
$acda$ is a triangle, denoted by $C_1$. As $G[U]\cong K_{n+1}$, let $C_2$ be any triangle
in $G[U]$. These two triangles are disjoint. In
$G-(V(C_1)\cup V(C_2))$, the remaining vertices of $U$ induce $K_{n-2}$, and the remaining vertices of $L$ induce $K_{n-4}$. Thus, the graph
$G-(V(C_1)\cup V(C_2))$ is matchable.
This contradicts the hypothesis that $G$ is a solid brick.

Finally, suppose that $n\ge5$ is odd. Then $|U|\ge6$ and $|L|\ge4$.
Let $T_U$ and $T_L$ be triangles in $G[U]-u_1$ and $G[L]-a$, respectively,
where $u_1\in N_G(a)\cap U$.
Since $G[U]\cong K_{n+1}$, the graph
$G[U]-(V(T_U)\cup\{u_1\})$ is a complete graph on $n-3$ vertices.
Since $G[L]$ is obtained from
$K_{n-1}$ by deleting $ab$, the graph
$G[L]-(V(T_L)\cup\{a\})$ is a complete graph on $n-5$ vertices.
 Let $M_U$ and $M_L$
be perfect matchings of these two graphs, respectively, where $M_L$ is empty
when $n=5$. Then
$M_U\cup M_L\cup\{u_1a\}$ is a perfect matching of
$G-(V(T_U)\cup V(T_L))$.
We deduce that $G$ is not solid, a contradiction. We conclude that Case~1 is
impossible.

\smallskip
\noindent\textbf{\upshape Case 2: $\Delta(G)\ge n+2$.}
Choose a vertex $x$ such that
\[
 d(x)=n+h,\qquad h\ge2.
\]
Let $A=N_G(x)$ be the set of neighbours of $x$ in $G$, and let
$B=V(G)-(A\cup\{x\})$ be the set of all vertices of $G$ other than $x$
that are not adjacent to $x$.
Since $d(x)=n+h$ and $G$ has order $2n$, we have
$|A|=n+h$ and $|B|=n-h-1$.
We assume that $A$ contains a vertex $v$ of degree at least $n+1$.
Corollary~\ref{cor:degree-pairs}(ii) would imply that $d(x)=n+1$, a
contradiction. Thus every vertex of $A$ has degree at most $n$. Similarly, if
some vertex $v\in B$ has degree at least $n$, then
Corollary~\ref{cor:degree-pairs}(i) would imply that $d(x)=n$, a contradiction.
Hence every vertex of $B$ has degree at most $n-1$.
Define
\[
 D=\sum_{v\in A}(n-d(v))+
   \sum_{v\in B}(n-1-d(v))\ge0.
\]
By adding up the degrees of all vertices of $G$ and using the definition of $D$, we
obtain
\begin{equation}
 \begin{aligned}
  2m
  &=d(x)+\sum_{v\in A}d(v)+\sum_{v\in B}d(v)\\
  &=d(x)+n|A|+(n-1)|B|-D\\
  &=n+h+n(n+h)+(n-1)(n-h-1)-D\\
  &=2n^2+2h-n+1-D.
 \end{aligned}
 \label{eq:defect-D}
\end{equation}
From \eqref{eq:defect-D} and \eqref{eq:brick-lower}, we
conclude that $0\le D\le2h-n-1$, and hence that
$h\ge\frac{n+1}{2}$. Since $n\ge4$, we have $h\ge3$.

Let $A_0=\{v\in A:d(v)=n\}$ denote the set of vertices in $A$ that have
degree $n$ in $G$. For every $v\in A-A_0$, we have $d(v)\le n-1$, and hence
$n-d(v)\ge1$. Therefore, by the definition of $D$,
\[
 |A-A_0|\le\sum_{v\in A}(n-d(v))\le D.
\]
Since $G$ is a simple graph of order $2n$, $n+h=d(x)\le2n-1$, and hence
$h\le n-1$. Since $|A|=n+h$, $|A-A_0|\le D$, and $D\le2h-n-1$, it
follows that
\[
 \begin{aligned}
  |A_0|
  &=|A|-|A-A_0|\\
  &\ge n+h-D\\
  &\ge n+h-(2h-n-1)\\
  &=2n-h+1\\
  &\ge n+2.
 \end{aligned}
\]
Choose $y\in A_0$; then $d(y)=n$ and $xy\in E(G)$. Hence $\eps=1$ by
\eqref{eq:epsilon}.
Let $M$ be a perfect matching of $G-x-y$.
By Lemma~\ref{lem:identity}, we obtain
\[
 s_x-r_x=(n+h)-1-(n-1)=h,\]
 and
 \[s_y-r_y=n-1-(n-1)=0.
\]
Since $h\ge3$ and $r_x\ge0$, we have $s_x=h+r_x\ge3$. Thus $M$ has at
least three edges with both end vertices adjacent to $x$. If $s_y>0$, then,
by Lemma~\ref{lem:common-edge}, we would have $s_x=1$, a contradiction.
Therefore
\begin{equation}
 s_y=r_y=0. \label{eq:sy-zero}
\end{equation}

Let $A_y=N_G(y)\cap V(G-x-y)=N_G(y)-\{x\}$, which is the set of neighbours of $y$
that belong to $V(G-x-y)$. Let $Q=V(G-x-y)-A_y$. It denotes the set of
vertices of $G-x-y$ that are not adjacent to $y$. Since $d(y)=n$ and
$xy\in E(G)$, the set $A_y$ has $n-1$ vertices. Since $G-x-y$ has $2n-2$
vertices, the set $Q$
also has $n-1$ vertices. By \eqref{eq:sy-zero}, no edge of $M$ has both endvertices in $A_y$, and no edge of $M$ has both endvertices in $Q$. Hence
every edge of $M$ has one endvertex in $A_y$ and the other endvertex in $Q$.

We claim that $Q$ is a stable set. Suppose that $q_1q_2\in E(G[Q])$. For
$i=1,2$, let $p_iq_i$ be the edge of $M$ incident with $q_i$. Then $p_i\in A_y$. Thus
\(
 yp_1q_1q_2p_2y
\)
is a cycle of length five, denoted by $C_1$. Since $s_x\ge3$, $M$ contains an edge
$p_3q_3$, distinct from $p_1q_1$ and $p_2q_2$, with both end vertices adjacent
to $x$. Adjust notation so that $p_3\in A_y$ and $q_3\in Q$. Let
$C_2=xp_3q_3x$. The triangle $C_2$ is disjoint from $C_1$, and
$M-\{p_1q_1,p_2q_2,p_3q_3\}$ is a perfect matching of
$G-(V(C_1)\cup V(C_2))$.
This contradicts the hypothesis that $G$ is a solid brick. Hence $Q$ is a stable set.

Since $Q$ is a stable set and $y$ is adjacent to no vertex of $Q$, the set
$Q\cup\{y\}$ is a stable set. Let $p$ be a vertex of $A_y$. Then $Q\cup\{y\}$ is a
stable set of $G-x-p$ with $n$ vertices. Since
$n>\frac{2n-2}{2}=\frac{|V(G-x-p)|}{2}$, the graph $G-x-p$ is not
matchable, contrary to the
bicriticality of $G$. Both cases are impossible. This completes the proof of
Theorem~\ref{thm:brick-bound}.
\end{proof}

%%%%%%%%%%%%%%%%%%%%%%%%%%%%%%%%%%%%%%%%%%%%%%%%%%%%%%%%%%%%%%%%%%%%%%%%%%%%%
\section{The maximum size of simple solid matching covered graphs}
%%%%%%%%%%%%%%%%%%%%%%%%%%%%%%%%%%%%%%%%%%%%%%%%%%%%%%%%%%%%%%%%%%%%%%%%%%%%

In this section, we first disprove Conjecture~\ref{conj:LM}. Then we correct the result of Conjecture~\ref{conj:LM}, and we prove that every simple solid
matching covered graph of order $2n$, for $n\ge2$, has at most $n^2+2$ edges. Moreover, we characterize the graphs attaining this bound.

\subsection{Disproof of Conjecture~\ref{conj:LM}}

Let $K_n$ be the complete graph on $n$ vertices, and let $\comp{K}_n$ be
the stable set on $n$ vertices. We now introduce three families of simple solid matching covered graphs, denoted by $A_n$, $B_n$ and $D_n$.
Let

\begin{gather*}
 A_n=\comp{K}_{n-1}\vee\bigl(K_3\cup(n-2)K_1\bigr) \qquad (n\ge2),\\
 B_n=\comp{K}_{n-2}\vee\bigl(2K_3\cup(n-4)K_1\bigr) \qquad (n\ge4)
\end{gather*}
and
\[
 D_n=\comp{K}_{n-2}\vee\bigl(F\cup(n-3)K_1\bigr) \qquad (n\ge3),
\]
where $F=K_1\vee(2K_2)$.
See Figure~\ref{fig:extremal} for the illustrations of $A_4$, $B_4$ and $D_4$.

\begin{figure}[!t]
\centering
\begin{tikzpicture}[
  scale=1.15,transform shape,
  vertex/.style={circle,draw=black,fill=white,inner sep=0pt,
    minimum size=6.2pt,line width=0.55pt},
  joinedge/.style={draw=black,solid,
    line width=0.55pt},
  inner/.style={draw=red,solid,line width=0.95pt},
  part/.style={font=\small}
]
% A_4
\begin{scope}[xshift=-5.25cm]
  \coordinate (AL1) at (-1.15,1.10);
  \coordinate (AL2) at (-1.15,0.00);
  \coordinate (AL3) at (-1.15,-1.10);
  \coordinate (AR1) at (0.34,1.32);
  \coordinate (AR2) at (1.25,0.88);
  \coordinate (AR3) at (0.34,0.44);
  \coordinate (AR4) at (0.88,-0.36);
  \coordinate (AR5) at (0.88,-1.20);
  \foreach \i in {1,2,3}{\foreach \j in {1,2,3,4,5}{\draw[joinedge] (AL\i)--(AR\j);}}
  \draw[inner] (AR1)--(AR2)--(AR3)--cycle;
  \foreach \i in {1,2,3}{\node[vertex] at (AL\i) {};}
  \foreach \i in {1,2,3,4,5}{\node[vertex] at (AR\i) {};}
\end{scope}
% B_4
\begin{scope}
  \coordinate (BL1) at (-1.15,0.68);
  \coordinate (BL2) at (-1.15,-0.68);
  \coordinate (BR1) at (0.34,1.30);
  \coordinate (BR2) at (1.25,0.92);
  \coordinate (BR3) at (0.34,0.54);
  \coordinate (BR4) at (0.34,-0.54);
  \coordinate (BR5) at (1.25,-0.92);
  \coordinate (BR6) at (0.34,-1.30);
  \foreach \i in {1,2}{\foreach \j in {1,2,3,4,5,6}{\draw[joinedge] (BL\i)--(BR\j);}}
  \draw[inner] (BR1)--(BR2)--(BR3)--cycle;
  \draw[inner] (BR4)--(BR5)--(BR6)--cycle;
  \foreach \i in {1,2}{\node[vertex] at (BL\i) {};}
  \foreach \i in {1,2,3,4,5,6}{\node[vertex] at (BR\i) {};}
\end{scope}
% D_4
\begin{scope}[xshift=5.25cm]
  \coordinate (CL1) at (-1.15,0.68);
  \coordinate (CL2) at (-1.15,-0.68);
  \coordinate (CR0) at (0.55,0.00);
  \coordinate (CR1) at (1.22,1.20);
  \coordinate (CR2) at (1.22,0.48);
  \coordinate (CR3) at (1.22,-0.48);
  \coordinate (CR4) at (1.22,-1.20);
  \coordinate (CR5) at (0.24,-1.34);
  \foreach \i in {1,2}{\foreach \j in {0,1,2,3,4,5}{\draw[joinedge] (CL\i)--(CR\j);}}
  \foreach \j in {1,2,3,4}{\draw[inner] (CR0)--(CR\j);}
  \draw[inner] (CR1)--(CR2);
  \draw[inner] (CR3)--(CR4);
  \foreach \i in {1,2}{\node[vertex] at (CL\i) {};}
  \foreach \i in {0,1,2,3,4,5}{\node[vertex] at (CR\i) {};}
\end{scope}
\end{tikzpicture}
\caption{The graphs $A_4$, $B_4$ and $D_4$.}
\label{fig:extremal}
\end{figure}
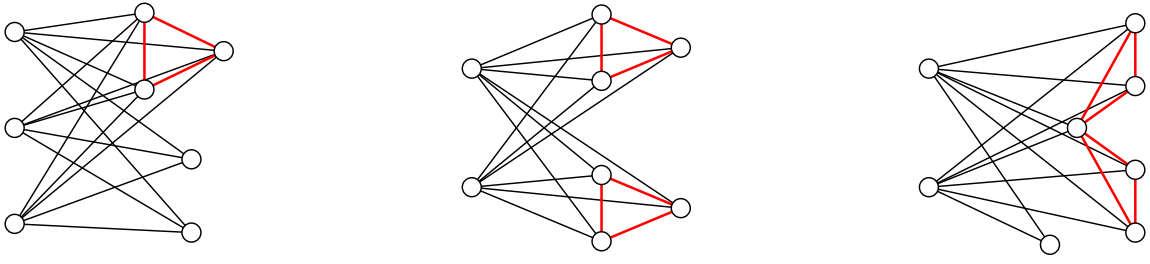

 We will show that each of $A_n$, $B_n$ and $D_n$ is a simple solid matching covered graph of order $2n$ and size $n^2+2$. Thus these families of graphs disprove
 Conjecture~\ref{conj:LM}. Moreover, for every integer $n\ge2$, we prove that each
 simple solid matching covered graph of order $2n$ has at most $n^2+2$ edges, and
 the graphs attaining the upper bound are exactly the graphs $A_n$, $B_n$ and $D_n$.
 To show these, We introduce some definitions and results.

A subset $B$ of the vertex set of a matchable graph $G$ is a \emph{barrier}
of $G$ if $o(G-B)=|B|$, where $o(G-B)$ denotes the number of odd components
of $G-B$.
If $B$ is a barrier of a matching covered graph $G$ and $K$ is an odd
component of $G-B$, then $\partial(V(K))$ is called a \emph{barrier cut}
associated with $B$.
A \emph{2-separation} of a matching covered graph $G$ is a 2-vertex-cut
$\{u,v\}$ that is not a barrier. As $G$ is matching covered, every component of $G-u-v$ has even order.
Suppose that $H_1$ is the union of a nonempty proper subset of the components
of $G-u-v$, and that $H_2$ is the union of the remaining components. The cuts
$\partial(V(H_1)\cup\{v\})$ and $\partial(V(H_2)\cup\{v\})$ are called
\emph{2-separation cuts}.

\begin{lemma}[\resultcite{tight-cuts}]\label{lem:barrier}
Let $B$ be a nonempty barrier of a matching covered graph $G$. Then $B$ is a
stable set, and every component of $G-B$ is odd.
\end{lemma}

Edmonds, Lov\'{a}sz and Pulleyblank proved the following result on barrier
cuts and 2-separation cuts.

\begin{lemma}[\resultcite{elp}]\label{lem:elp-cuts-tight}
Let $C$ be either a barrier cut or a 2-separation cut of a matching covered
graph $G$. Then $C$ is a tight cut of $G$.
\end{lemma}

\begin{lemma}\label{lem:order-four}
The following statements hold.
\begin{enumerate}
\item[(i)] If $G$ is a bipartite matching covered graph of order four, then
$\operatorname{Sim}(G)\cong C_4$.
\item[(ii)] If $G$ is a nonbipartite matching covered graph of order four,
then $\operatorname{Sim}(G)\cong K_4$.
\item[(iii)] Every matching covered graph of order four is solid.
\end{enumerate}
\end{lemma}

\begin{proof}
Let $G$ be a bipartite matching covered graph of order four, and let $(A,B)$
be its bipartition. Since $G$ has a perfect matching, $|A|=|B|=2$. By
Lemma~\ref{lem:two-connected}, $G$ is 2-connected, and hence every vertex of
$G$ has at least two distinct neighbours. It follows that every vertex in
$A$ is adjacent to both vertices in $B$. Thus
$\operatorname{Sim}(G)\cong C_4$. This proves part~(i).

Let $G$ be a nonbipartite matching covered graph of order four. It contains a
triangle $abc$; let $d$ be the fourth vertex. Let $M$ be a perfect matching
containing $ab$. Then $cd\in M$, and hence $cd\in E(G)$. A perfect matching
containing $bc$ must also contain $ad$, and a perfect matching containing
$ca$ must also contain $bd$. Thus $\operatorname{Sim}(G)\cong K_4$. This
proves part~(ii).

To prove (iii), let $G$ be a matching covered graph of order four, and let
$C$ be a separating cut of $G$. Since both $C$-contractions are matching
covered, $C$ is an odd cut. Since $G$ has order four, $C$ has a singleton
shore and hence $C$ is  a tight cut. Therefore, $G$ is solid.
\end{proof}

For a matching covered graph with a nontrivial tight cut, we have the
following result.

\begin{theorem}[\resultcite{clm-characteristic}]\label{thm:solid-tight-contractions}
Let $C$ be a tight cut of a matching covered graph $G$. The graph $G$ is solid
if and only if both $C$-contractions of $G$ are solid.
\end{theorem}

As shown in Figure~\ref{fig:extremal}, $A_n$ is the join of $\comp{K}_{n-1}$
and $K_3\cup\comp{K}_{n-2}$; $B_n$ is the join of $\comp{K}_{n-2}$ and
$2K_3\cup\comp{K}_{n-4}$; and $D_n$ is the join of $\comp{K}_{n-2}$ and
$F\cup\comp{K}_{n-3}$, where $F$ is the join of $K_1$ and $2K_2$.
In the following, we will show that the graphs $A_n$, $B_n$ and $D_n$ are
simple solid matching covered graphs.

\begin{lemma}\label{lem:families}
The graphs $A_n$ ($n\ge2$), $B_n$ ($n\ge4$) and $D_n$ ($n\ge3$) are
simple solid matching covered graphs. Each of them has order $2n$ and size $n^2+2$.
\end{lemma}

\begin{proof}
By the definitions, the graphs $A_n$, $B_n$ and $D_n$ are simple and
connected, and each has order $2n$ and size $n^2+2$.

We first consider $A_n$. We begin by proving that $A_n$ is matching covered.
In the definition of $A_n$,
let $H=K_3\cup(n-2)K_1$, so that $A_n=\comp{K}_{n-1}\vee H$, and let $T$
denote the triangle in $H$. The graph $A_2=K_4$ is matching covered.
Suppose that $n\ge3$, and let $e$ be any edge of $A_n$. We consider separately
the cases in which $e$ has both endvertices, precisely one endvertex or no endvertex in $V(T)$.
Suppose first that $e$ has both endvertices in $V(T)$, and write $e=ab$. Then
the graph $A_n-a-b$ is isomorphic to $K_{n-1,n-1}$. Let $M$ be a perfect
matching of $A_n-a-b$.
Then $M\cup\{e\}$ is a perfect matching of $A_n$ containing $e$.
Suppose next that $e$ has precisely one endvertex in $V(T)$. Let $a$ be the
endvertex of
$e$ in $V(T)$, and let $x$ be its other endvertex. By the definition of $A_n$,
$x\in V(A_n)-V(H)$. Let $bc$ be the edge of $T-a$.
The graph $A_n-\{a,x,b,c\}$ is isomorphic to $K_{n-2,n-2}$ and hence has a
perfect matching $M$. Then $M\cup\{e,bc\}$ is a perfect matching of $A_n$
containing $e$. Finally, suppose that $e$ has no endvertex in $V(T)$. By the
definition of $A_n$, we may write $e=xy$, where $x\in V(A_n)-V(H)$ and
$y\in V(H)-V(T)$. Let $ab$ be any edge of $T$. The graph
$A_n-\{x,y,a,b\}$ is isomorphic to $K_{n-2,n-2}$ and hence has a perfect
matching $M$. Then $M\cup\{e,ab\}$ is a perfect matching of $A_n$ containing
$e$.
Thus every edge of $A_n$ is covered by a perfect matching. Since $A_n$ is
connected, it is a matching covered graph.

We next prove that $A_n$ is solid. For $n=2$, we have $A_2=K_4$, which is
solid by Lemma~\ref{lem:order-four}(iii). Suppose that $n\ge3$, and let
$C=\partial(V(T))$. By the definition of $A_n$, $C$ is the set of
edges joining $V(T)$ and $V(A_n)-V(H)$. Let $M$ be any perfect matching of
$A_n$. Since $V(A_n)-V(H)$ is a stable set of cardinality $n-1$, the matching $M$
contains precisely $n-1$ edges joining $V(A_n)-V(H)$ and $V(H)$. Moreover,
since every vertex of $V(H)-V(T)$ is isolated in $H$, the matching $M$
contains precisely $n-2$ edges joining $V(H)-V(T)$ and $V(A_n)-V(H)$.
Therefore,
\[
 |M\cap C|=(n-1)-(n-2)=1.
\]
Since $M$ was arbitrary, $C$ is a tight cut of $A_n$.
Since $C$ is tight, both $C$-contractions are matching covered. The underlying
simple graphs of these contractions are isomorphic to $K_4$ and
$K_{n-1,n-1}$. The graph $K_4$ is solid by Lemma~\ref{lem:order-four}(iii),
whereas $K_{n-1,n-1}$ is solid by Theorem~\ref{thm:bipartite-solid}. Thus, both
$C$-contractions are solid by Lemma~\ref{lem:simplification}. It follows from
Theorem~\ref{thm:solid-tight-contractions} that $A_n$ is solid.

We next consider $B_n$. We begin by proving that $B_n$ is matching covered.
In the definition of $B_n$, let
$J=2K_3\cup(n-4)K_1$, so that $B_n=\comp{K}_{n-2}\vee J$, and let $T_1$
and $T_2$ denote the two triangles in $J$. Let $e$ be any edge of $B_n$.
We consider separately the cases in which $e$ has both end vertices in one of the
triangles, precisely one end vertex in $V(T_1)\cup V(T_2)$ or no end vertex in
$V(T_1)\cup V(T_2)$.
Suppose first that $e$ has both end vertices in one of the triangles. Adjust notation
so that $e=ab\in E(T_1)$. The graph $B_n-a-b$ is isomorphic to $A_{n-1}$,
which is matching covered by the preceding argument. Let $M$ be a perfect
matching of $B_n-a-b$. Then $M\cup\{e\}$ is a perfect matching of $B_n$
containing $e$.
Suppose next that $e$ has precisely one endvertex in $V(T_1)\cup V(T_2)$.
Adjust notation so that this endvertex, say $a$, lies in $V(T_1)$, and write
$e=ax$, where $x\in V(B_n)-V(J)$. Let $bc$ be the edge of $T_1-a$.
The graph $B_n-\{a,x,b,c\}$ is isomorphic to $A_{n-2}$ and hence has a
perfect matching $M$. Then $M\cup\{e,bc\}$ is a perfect matching of $B_n$
containing $e$.
Finally, suppose that $e$ has no endvertex in $V(T_1)\cup V(T_2)$. By the
definition of $B_n$, we may write $e=xy$, where $x\in V(B_n)-V(J)$ and $y$
is an isolated vertex of $J$; in particular, $n\ge5$. Let $ab$ be any edge
of $T_1$. The graph $B_n-\{x,y,a,b\}$ is isomorphic to $A_{n-2}$ and hence
has a perfect matching $M$. Then $M\cup\{e,ab\}$ is a perfect matching of
$B_n$ containing $e$. Thus every edge of $B_n$ is covered by a perfect
matching. Since $B_n$ is connected, it is a matching covered graph.

We next prove that $B_n$ is solid. Let $M$ be any perfect matching of $B_n$.
Since $V(B_n)-V(J)$ is a stable set of cardinality $n-2$, the matching $M$ contains
precisely $n-2$ edges joining $V(B_n)-V(J)$ and $V(J)$. The $n-4$ isolated
vertices of $J$ are matched by $M$ to vertices of $V(B_n)-V(J)$. Consequently,
precisely two edges of $M$ join $V(B_n)-V(J)$ and
$V(T_1)\cup V(T_2)$. Since neither triangle has a perfect matching, at least
one of these two edges is incident with each triangle. Thus
$|M\cap C|=1$, where $C=\partial(V(T_1))$. Since $M$ is arbitrary,
$C$ is a tight cut of $B_n$.
Since $C$ is tight, both $C$-contractions are matching covered. The underlying
simple graphs of these contractions are isomorphic to $K_4$ and $A_{n-1}$.
The graph $K_4$ is solid by Lemma~\ref{lem:order-four}(iii), whereas
$A_{n-1}$ is solid by the preceding argument. Thus both $C$-contractions are
solid by Lemma~\ref{lem:simplification}. It follows from
Theorem~\ref{thm:solid-tight-contractions} that $B_n$ is solid.

Finally, we consider $D_n$. We begin by proving that $D_n$ is matching
covered. In the definition of $D_n$, let
$Q=F\cup(n-3)K_1$, so that $D_n=\comp{K}_{n-2}\vee Q$.
Let $z$ denote the vertex of degree four in $F$,
and let $a_1b_1$ and $a_2b_2$ denote the two edges of $F-z$. Let $e$ be any
edge of $D_n$. We distinguish five cases according to the number of endvertices of
$e$ in $V(F)$ and whether $z$ is an endvertex of $e$.
If $e$ has both endvertices in $V(F)$ and $z$ is an endvertex of
$e$, adjust notation so that $e=za_1$. The graph
$D_n-\{z,a_1,a_2,b_2\}$ is isomorphic to $K_{n-2,n-2}$ and hence has a
perfect matching $M$. Then $M\cup\{e,a_2b_2\}$ is a perfect matching of
$D_n$ containing $e$.
Suppose that $e$ has both endvertices in $V(F)$ and $z$ is not an endvertex of
$e$. Adjust notation so that $e=a_1b_1$. The graph $D_n-a_1-b_1$ is
isomorphic to $A_{n-1}$, which is matching covered by the preceding argument.
Let $M$ be a perfect matching of $D_n-a_1-b_1$. Then $M\cup\{e\}$ is a
perfect matching of $D_n$ containing $e$.
If $e$ has precisely one endvertex in $V(F)$ and this endvertex is
$z$, write $e=zx$, where $x\in V(D_n)-V(Q)$. The graph
$D_n-\{z,x,a_1,b_1,a_2,b_2\}$ is isomorphic to $K_{n-3,n-3}$. Let $M$ be
a perfect matching of this graph, taking $M$ to be the empty matching when
$n=3$. Then
$M\cup\{e,a_1b_1,a_2b_2\}$ is a perfect matching of $D_n$ containing $e$.
If $e$ has precisely one endvertex in $V(F)$ and this endvertex is
not $z$, adjust notation so that $e=a_1x$, where $x\in V(D_n)-V(Q)$. The
graph $D_n-\{a_1,x,z,b_1,a_2,b_2\}$ is isomorphic to $K_{n-3,n-3}$ and
hence has a perfect matching $M$, taking $M$ to be the empty matching when
$n=3$. Then $M\cup\{e,zb_1,a_2b_2\}$ is a perfect matching of $D_n$
containing $e$.
Finally, suppose that $e$ has no endvertex in $V(F)$. By the definition of $D_n$,
we may write $e=xy$, where $x\in V(D_n)-V(Q)$ and $y$ is an isolated vertex
of $Q$; in particular, $n\ge4$. The graph $D_n-\{x,y,a_1,b_1\}$ is
isomorphic to $A_{n-2}$ and hence has a perfect matching $M$. Then
$M\cup\{e,a_1b_1\}$ is a perfect matching of $D_n$ containing $e$. Thus
every edge of $D_n$ is covered by a perfect matching. Since $D_n$ is
connected, it is a matching covered graph.

We next prove that $D_n$ is solid. We first consider the case $n=3$. Let $u$
denote the unique vertex of
$V(D_3)-V(F)$. Both $u$ and $z$ are adjacent to every other vertex of $D_3$,
and $D_3-u-z$ has two components, each isomorphic to $K_2$. Thus, $\{u,z\}$
is a 2-vertex-cut of $D_3$. Since both components of $D_3-u-z$ have even
order, $\{u,z\}$ is not a barrier of $D_3$. It follows that $\{u,z\}$ is a
2-separation of $D_3$. Let $X$ and $Y$ denote the vertex sets of the two
components of $D_3-u-z$.
Set
\[
 C=\partial(X\cup\{z\})=\partial(Y\cup\{u\}).
\]
Then $C$ is a 2-separation cut of $D_3$ and hence is tight by
Lemma~\ref{lem:elp-cuts-tight}. Since $C$ is tight, its two contractions are
matching covered. Each contraction has order four and is therefore solid by
Lemma~\ref{lem:order-four}(iii). It follows from
Theorem~\ref{thm:solid-tight-contractions} that $D_3$ is solid.

Now suppose that $n\ge4$, and set $C=\partial(V(F))$. In every
perfect matching of $D_n$, the $n-3$ isolated vertices of $Q$ are matched to
$n-3$ vertices of $V(D_n)-V(Q)$, and the remaining vertex of
$V(D_n)-V(Q)$ is matched to a vertex of $F$. Thus, every perfect matching of
$D_n$ contains precisely one edge of $C$, and hence $C$ is tight.
Since $C$ is tight, its two contractions are matching covered. The underlying
simple graphs of these contractions are isomorphic to $D_3$ and
$K_{n-2,n-2}$. The graph $D_3$ is solid by the preceding paragraph, and
$K_{n-2,n-2}$ is solid by Theorem~\ref{thm:bipartite-solid}. Thus both contractions are
solid by Lemma~\ref{lem:simplification}. It follows from
Theorem~\ref{thm:solid-tight-contractions} that $D_n$ is solid.
\end{proof}

Each of the three families can contains graphs of arbitrarily large order. Therefore,
Lemma~\ref{lem:families} disproves Conjecture~\ref{conj:LM}.

\subsection{Proof of Theorem~\ref{thm:main}}

Here, we will prove our main result, Theorem~\ref{thm:main}. We first introduce a lemma.

\begin{lemma}\label{lem:small-bricks}
Every simple solid brick of order six has at most ten edges.
\end{lemma}

\begin{proof}
Let $G$ be a simple solid brick of order six. By
Theorem~\ref{thm:brick-characterization}, $G$ is 3-connected.
Thus $\delta(G)\ge3$, and hence $\Delta(\comp G)\le2$. Suppose that
$|E(G)|\ge11$. Then $|E(\comp G)|\le4$.

Let $S$ be a subset of $V(G)$ of cardinality three. Then $|V(G)-S|=3$. If
$S$ induces a triangle in $\comp G$, then $S$ is a stable set in $G$. Thus, no
edge of $G$ with both endvertices in $V(G)-S$ lies in a perfect matching of $G$.
Since $G$ is matching covered, $V(G)-S$ is also a stable set in $G$. Hence $G$
is bipartite, contrary to the fact that $G$ is a brick. Therefore, $\comp G$
contains no triangle.

Since $\comp G$ has at most four edges, every odd cycle in $\comp G$ would be
a triangle. Thus $\comp G$ has no odd cycle. Since
$\Delta(\comp G)\le2$, each component of $\comp G$ is a path or an even cycle.
Every component of $\comp G$ of even order is either an even cycle or a path of
even order, and hence has a bipartition with parts of equal cardinality.
Every component of $\comp G$ of odd order is a path and has a bipartition whose
parts differ in cardinality by one. Since $\comp G$ has even order, the number
of its components of odd order is even. Therefore, there exists a bipartition
$(X,Y)$ of $\comp G$ such that $|X|=|Y|$.
Since $\comp G$ has order six, $|X|=|Y|=3$. Thus, $X$ and $Y$ are stable sets
of $\comp G$, each of cardinality three. Therefore, the subgraphs $G[X]$ and
$G[Y]$ are complete. Hence $G[X]$ and $G[Y]$ are two vertex-disjoint
triangles of $G$; let $C_1$ and $C_2$ denote these two triangles. Since
$V(C_1)\cup V(C_2)=V(G)$, the graph
$G-(V(C_1)\cup V(C_2))$ is empty, and the empty set is a perfect matching of
this graph. By Theorem~\ref{thm:odd-bicycle}, $G$ is nonsolid, contrary to the
hypothesis that $G$ is solid. Hence
$|E(G)|\le10$.
\end{proof}

We use the term \emph{ELP cut} to mean either a barrier cut or a
2-separation cut. ELP cuts play an important role in tight cut decompositions
of matching covered graphs. The following theorem gives a sufficient condition
for the existence of a nontrivial ELP cut.

\begin{theorem}[\resultcite{elp}]\label{thm:ELP}
If a matching covered graph has a nontrivial tight cut, then it has a
nontrivial ELP cut.
\end{theorem}

\begin{proposition}\label{prop:small-equality}
The following statements hold.
\begin{enumerate}
\item[(i)] Let $G$ be a simple solid matching covered graph of order six. Then
$|E(G)|\le11$. Equality holds if and only if $G\cong A_3$ or $G\cong D_3$.
\item[(ii)] Let $G$ be a simple solid matching covered graph of order eight. If
$|E(G)|=18$, then $G\cong A_4$, $G\cong B_4$ or $G\cong D_4$.
\end{enumerate}
\end{proposition}

\begin{proof}
(i)
If $G$ is bipartite, then the two parts of its bipartition have cardinality
three, and hence $|E(G)|\le9$. If $G$ is a brick, then
Lemma~\ref{lem:small-bricks} gives $|E(G)|\le10$. We may therefore assume
that $G$ is nonbipartite and is not a brick. It then has a nontrivial
tight cut. Then, by Theorem~\ref{thm:ELP}, $G$ has a nontrivial ELP cut,
denoted by $C$. Then $C$ is either a nontrivial barrier cut or a nontrivial 2-separation cut.

Suppose first that $C$ is a nontrivial barrier cut. Let $B$ denote a barrier of
$G$ such that $C$ is associated with $B$.
Set $k=|B|$. By Lemma~\ref{lem:barrier}, $B$ is a stable set and $G-B$ has
exactly $k$ odd components. Since $C$ is nontrivial and
$|V(G-B)|=6-k\ge k$, we obtain that $2\le k\le3$.
Consider first that $k=2$. The two components of $G-B$ have orders
three and one, respectively.
Let $Q$ denote the component of order three. Since $B$ is a stable set and no
edge joins vertices in distinct components of $G-B$, every edge of $G$ either
lies in $Q$ or joins a vertex of $B$ to a vertex of $G-B$. Since $G$ is simple,
$Q$ has at most three edges, and at most eight edges join $B$ and $V(G-B)$.
Hence $|E(G)|\le11$. If equality holds, then $Q\cong K_3$, and every vertex of
$B$ is adjacent to every vertex of $G-B$. Therefore,
\[
 G\cong\comp{K}_2\vee(K_3\cup K_1)=A_3.
\]
If $k=3$, then every component of $G-B$ is trivial. Since $B$ is a stable set
and no edge joins vertices in distinct components of $G-B$, every edge of $G$
joins a vertex of $B$ to a vertex of $G-B$. Since $G$ is simple,
$|E(G)|\le9$ in this case.

Now suppose that $C$ is a nontrivial 2-separation cut
associated with a
2-separation $\{u,v\}$. Since each component of $G-u-v$ is even and $G-u-v$
has order four, it has precisely two components whose vertex sets are denoted
by $X$ and $Y$. Then $|X|=|Y|=2$. Since $G$ is simple and $|X|=|Y|=2$, we have
\[
 |E(G)|=|E(G[X])|+|E(G[Y])|+|E[X\cup Y,\{u,v\}]|+|E(G[\{u,v\}])|\le11.
\]
If equality holds, then $G[X]\cong G[Y]\cong K_2$, every vertex of
$X\cup Y$ is adjacent to both $u$ and $v$, and $uv\in E(G)$. Therefore,
\[
 G\cong K_2\vee(2K_2)=D_3.
\]
Conversely, Lemma~\ref{lem:families} shows that $A_3$ and $D_3$ are simple
solid matching covered graphs of order six and size eleven. This proves (i).

 (ii)
If $G$ is bipartite, then $|E(G)|\le16$, a contradiction. If $G$ is a brick,
then Theorem~\ref{thm:brick-bound} shows that $|E(G)|\le16$, a contradiction.
Thus, $G$ is nonbipartite and is not a brick. It therefore has a nontrivial
tight cut. Then, by Theorem~\ref{thm:ELP}, $G$ has a nontrivial ELP cut,
denoted by $C$.

If $C$ is a nontrivial barrier cut. Let $B$ be the barrier of
$G$ such that $C$ is associated with $B$.
Let $k=|B|$. By Lemma~\ref{lem:barrier}, $B$ is a stable set and $G-B$ has
exactly $k$ odd components. Since $C$ is nontrivial and
$|V(G-B)|\ge k$, we obtain $2\le k\le4$.
If $k=4$, then every component of $G-B$ is trivial. Since $B$ is a stable set
and $G$ is simple, $|E(G)|\le16$, a contradiction. Hence $2\le k\le3$.

If $k=2$. The two components of $G-B$ either both have order three
or have orders five and one. Suppose first that both components have order three. Since $B$ is a stable set
and no edge of $G$ joins vertices in distinct components of $G-B$, every edge
of $G$ either lies in one of the two components of $G-B$ or joins a vertex of
$B$ to a vertex of $G-B$. As $G$ is simple, each component has at most three
edges. Moreover, since $|B|=2$ and $|V(G-B)|=6$, at most twelve edges join a
vertex of $B$ to a vertex of $G-B$. Hence $|E(G)|\le18$.
Since $|E(G)|=18$, both components are triangles,
and every vertex of $B$ is adjacent to every vertex of $G-B$. Therefore,
\[
 G\cong\comp{K}_2\vee(2K_3)=B_4.
\]

Suppose next that the components of $G-B$ have orders five and one. Let $Q$
denote the component of order five, let $w$ denote the vertex of the trivial
component, and let $B=\{x,y\}$. By Lemma~\ref{lem:two-connected}, $G$ is
2-connected, and hence $d(w)\ge2$. Since $w$ has no neighbour outside $B$ and $|B|=2$, we have $N_G(w)=B$.
The cut $\partial(V(Q))$ is a barrier cut associated with $B$. By
Lemma~\ref{lem:elp-cuts-tight}, it is tight and hence it is a separating cut. Let $H$ be the graph obtained by contracting $B\cup\{w\}$ to a
single vertex $z$. Since the cut is a separating cut, $H$ is matching covered. Since
$G$ is solid, Theorem~\ref{thm:solid-tight-contractions} implies that $H$ is
solid. Let $S=\operatorname{Sim}(H)$ be the underlying simple graph of $H$.
Lemma~\ref{lem:simplification} implies that $S$ is a simple solid matching covered graph of order six.

Since $B$ is a stable set and $N_G(w)=B$, the only edges of $G$ with both
endvertices in $B\cup\{w\}$ are $xw$ and $yw$. Hence $|E(H)|=|E(G)|-2=16$.
By the definition of $H$, we have $H-z=Q$. Since $G$ is simple, $Q$ has no
parallel edges, and hence $S-z=Q$. Let $q\in N_S(z)$. The number of edges joining $q$ and $z$ in $H$ is the number of edges
 joining $q$ to $B\cup\{w\}$ in $G$. Since $q$ and $w$ lie in distinct
components of $G-B$, we have $qw\notin E(G)$. Since $B=\{x,y\}$ and $G$ is
simple, at most two edges of $H$ join $q$ and $z$. This conclusion holds for
each $q\in N_S(z)$. Consequently,
\[
 d_H(z)\le2|N_S(z)|=2d_S(z).
\]

By the result of part~(i), $|E(S)|\le11$. Since $S$ has order six, $d_S(z)\le5$.
Consequently,
\[
 \begin{aligned}
 16=|E(H)|
 &=|E(H-z)|+d_H(z)\\
 &=|E(S-z)|+d_H(z)\\
 &\le |E(S-z)|+2d_S(z)\\
 &=|E(S)|+d_S(z)\\
 &\le11+5=16.
 \end{aligned}
\]
Equality holds throughout. Thus $|E(S)|=11$, $d_S(z)=5$, and
$d_H(z)=2d_S(z)$. It follows from part~(i) that
$S\cong A_3$ or $S\cong D_3$. Since $d_S(z)=5$ and $\Delta(A_3)=4$, we have
$S\cong D_3$. By the definition of $D_3$, $Q=S-z$, and so $Q\cong F$.
Since $S-z=Q$ and $d_S(z)=5=|V(Q)|$, we have $N_S(z)=V(Q)$. Since equality
holds in $d_H(z)\le2|N_S(z)|$, precisely two edges of $H$ join $q$ and $z$
for every $q\in V(Q)$. For every $q\in V(Q)$, the definition of $H$ implies
that precisely two edges of $G$ join $q$ to $B\cup\{w\}$. Since $G$ is simple,
$qw\notin E(G)$, and $B=\{x,y\}$,  every vertex of $B$ is adjacent to every vertex of $Q$. Together with $N_G(w)=B$, this shows that every vertex of $B$ is adjacent to every vertex of $G-B$. Since $B$ is a stable set, $G[B]\cong\comp{K}_2$.
Moreover, since $Q\cong F$ and $\{w\}$ is a component of $G-B$, we have
$G-B\cong F\cup K_1$. Therefore,
\[
 G\cong\comp{K}_2\vee(F\cup K_1)=D_4.
\]

Suppose that $k=3$. Since $G-B$ has exactly three odd components and
$|V(G-B)|=5$, one component has order three and the other two are trivial.
Let $Q$ denote the component
of order three. Since $B$ is a stable set and no edge of $G$ joins vertices in
distinct components of $G-B$, every edge of $G$ either lies in $Q$ or joins a
vertex of $B$ to a vertex of $G-B$. Since $G$ is simple, $Q$ has at most three
edges, and $G$ has at most fifteen edges join a vertex of $B$ to a vertex of $G-B$.
Hence
\[
 18=|E(G)|=|E(Q)|+|E[B,V(G-B)]|\le3+15=18.
\]
Equality holds throughout. Thus $|E(Q)|=3$ and
$|E[B,V(G-B)]|=15$. Hence $Q\cong K_3$, and every vertex of $B$ is adjacent
to every vertex of $G-B$. Since the other two components of $G-B$ are trivial,
$G-B\cong K_3\cup2K_1$. Moreover, since $B$ is a stable set of cardinality
three, $G[B]\cong\comp{K}_3$. Therefore,
\[
 G\cong\comp{K}_3\vee(K_3\cup2K_1)=A_4.
\]
Now, we consider the case that $C$ is a nontrivial 2-separation cut associated with a
2-separation $\{u,v\}$. Let $X$ and $Y$ be nonempty unions of components of
$G-u-v$ such that $X\cup Y=V(G)-\{u,v\}$ and $X\cap Y=\emptyset$. Then
$C=\partial(X\cup\{u\})=\partial(Y\cup\{v\})$.
Since every component of $G-u-v$ has even order and $|X|+|Y|=6$, we may
assume that $|X|=2$ and $|Y|=4$.
Let $H_X$ be the
 graph obtained from $G$ by
contracting $Y\cup\{v\}$ to a single vertex, and let $H_Y$ be the graph obtained by
contracting $X\cup\{u\}$ to a single vertex. Thus, $H_X$ and $H_Y$ are
two $C$-contractions of $G$. Let
$S_X=\operatorname{Sim}(H_X)$ and $S_Y=\operatorname{Sim}(H_Y)$.
The cut $C$ is a 2-separation cut and hence is tight by
Lemma~\ref{lem:elp-cuts-tight}. Since every tight cut is a separating cut, $H_X$
and $H_Y$ are both matching covered graphs. Since $G$ is solid,
Theorem~\ref{thm:solid-tight-contractions} implies that $H_X$ and $H_Y$ are
solid. By Lemma~\ref{lem:simplification}, $S_X$ and $S_Y$ are simple solid
matching covered graphs of orders four and six, respectively. By
Lemma~\ref{lem:two-connected}, $G$ is 2-connected. Thus both $u$ and $v$ have
a neighbour in each of $X$ and $Y$. Otherwise, after possibly interchanging
$u$ and $v$, we may assume that $u$ has no neighbour in one of $X$ and $Y$;
then $v$ is a cut vertex of $G$, a contradiction.
Let $y$ and $x$ denote the contraction vertices of $H_X$ and
$H_Y$, respectively. Since $G$ is simple and no edge of $G$ joins a vertex of
$X$ to a vertex of $Y$, the graph $S_X$
is obtained from $H_X$ by replacing the parallel edges with endvertices $u$
and $y$ by the single edge $uy$. Similarly, $S_Y$ is obtained from $H_Y$ by
replacing the parallel edges with endvertices $v$ and $x$ by the single edge
$vx$. Hence
\[
 |E(S_X)|=|E(G[X])|+|E[X,\{u,v\}]|+1
\]
and
\[
 |E(S_Y)|=|E(G[Y])|+|E[Y,\{u,v\}]|+1.
\]
Since no edge of $G$ joins a vertex of $X$ to a vertex of $Y$, and since
$G$ is simple, we have
\[
 \begin{aligned}
 |E(G)|
 &\le |E(G[X])|+|E(G[Y])|+|E[X,\{u,v\}]|
      +|E[Y,\{u,v\}]|+1\\
 &=\bigl(|E(S_X)|-1\bigr)+\bigl(|E(S_Y)|-1\bigr)+1\\
 &=|E(S_X)|+|E(S_Y)|-1\\
 &\le6+11-1=16,
 \end{aligned}
\]
where the last inequality follows since $S_X$ is a simple graph of order four
and, by part~(i), $|E(S_Y)|\le11$. This is a contradiction.
\end{proof}

Now, we are ready to prove Theorem~\ref{thm:main}.

\begin{proof}[\textbf{\upshape Proof of Theorem~\ref{thm:main}}]
Let $G$ be a simple solid matching covered graph of order $2n$. We will prove by
induction on $n$. For $n=2$, every simple bipartite graph of order four has at most $4$ edges. Thus, $|E(G)|\le n^2+2$. By Lemma~\ref{lem:order-four}(ii), the only simple
nonbipartite matching covered graph of order four is $K_4$. The number of edges in $K_4$ is six. Moreover, $K_4=A_2$. Thus, the theorem holds for $n=2$.

For the case in which $n\ge3$, assume that the theorem holds for every integer $m$ with
$2\le m<n$. Let $G$ be a simple solid matching covered graph of order $2n$.
If $G$ is bipartite, then the two parts of its bipartition have the same
cardinality, and both parts have $n$
vertices. Since $G$ is simple,
\[
 |E(G)|\le n^2<n^2+2.
\]
Hence the theorem holds when $G$ is bipartite.

Suppose that $G$ is nonbipartite. If $G$ is a brick of order six, then
Lemma~\ref{lem:small-bricks} gives that $|E(G)|\le10$, and $10\le11=n^2+2$
when $n=3$. If $G$ is a brick of order $2n$, for $n\ge4$,
Theorem~\ref{thm:brick-bound} shows that $|E(G)|\le n^2<n^2+2$.
Hence the theorem holds when $G$ is a brick. We may thus assume
that $G$ is not a brick. Since $G$ is a nonbipartite matching covered graph, the
definition of a brick implies that $G$ has a nontrivial tight cut. By
Theorem~\ref{thm:ELP}, $G$ has a nontrivial ELP cut, denoted by $C$.

%\smallskip

\noindent\textbf{\upshape Case 1:}
The ELP cut $C$ is a nontrivial barrier cut. Let $B$ be a barrier of $G$
such that $C$ is
associated with $B$. By Lemma~\ref{lem:barrier}, $B$ is a stable set and every
component of $G-B$ is odd. Since $C$ is nontrivial, $B$ is a nontrivial
barrier. Let $k=|B|$. Then $k\ge2$. Since $B$ is a barrier,
$o(G-B)=|B|=k$. As every component of $G-B$ is odd, $G-B$ has exactly
$k$ components. Let $Q_1,\ldots,Q_k$ be the components of $G-B$.
Set
\[
 |V(Q_i)|=2n_i-1,\qquad \text{for }1\le i\le k,
\]
where $n_i\ge1$. The sets $B,V(Q_1),\ldots,V(Q_k)$ form a
partition of $V(G)$. Counting vertices in
this partition gives
$2n=k+\sum_{i=1}^{k}(2n_i-1)=2\sum_{i=1}^{k}n_i$, and hence
\begin{equation}
 \sum_{i=1}^{k} n_i=n. \label{eq:sum-ni}
\end{equation}
Since no edge of $G$ joins vertices in different components of $G-B$ and no
edge of $G$ has both end vertices in $B$, we have
\begin{equation}
 |E(G)|=\sum_{i=1}^{k}\bigl(|E(Q_i)|+|E[V(Q_i),B]|\bigr).
 \label{eq:barrier-partition}
\end{equation}

Each $Q_i$ is an odd component of $G-B$, so
$\partial(V(Q_i))$ is a barrier cut associated with $B$. By
Lemma~\ref{lem:elp-cuts-tight}, this cut is tight and hence it is a separating cut. Let
$H_i$ be the graph obtained by contracting
$V(G)-V(Q_i)$ to a single vertex, denoted by $c_i$. Since the cut $\partial(V(Q_i))$ is a separating cut, $H_i$ is
matching covered, for $1\le i\le k$. Since $G$ is solid, $H_i$ is solid by
Theorem~\ref{thm:solid-tight-contractions}. Let
$S_i=\operatorname{Sim}(H_i)$ be the underlying simple graph of $H_i$. Since
$|V(Q_i)|=2n_i-1$, both $H_i$ and $S_i$ have order $2n_i$. By
Lemma~\ref{lem:simplification}, $S_i$ is a solid matching covered graph. Define
\[
 p_i=|\{x\in V(Q_i):N_G(x)\cap B\ne\varnothing\}|,
 \qquad \text{for }1\le i\le k.
\]
Thus, precisely $p_i$ vertices of $Q_i$ have neighbours in $B$.
Since $Q_i$ is a component of $G-B$, every edge of $H_i$ outside $E(Q_i)$
joins a vertex of $Q_i$ to the contraction vertex. For $x\in V(Q_i)$ with
$N_G(x)\cap B\ne\varnothing$, there is an edge between $x$ and the contraction
vertex $c_i$; for $x\in V(Q_i)$ with $N_G(x)\cap B=\varnothing$, there is no
edge between $x$ and the contraction vertex $c_i$. Since $G$ is simple, no two edges of
$Q_i$ are parallel, so every edge of $Q_i$ is retained in $S_i$.
Consequently,
\begin{equation}
 |E(S_i)|=|E(Q_i)|+p_i,\qquad \text{for }1\le i\le k.
 \label{eq:si-edge-count}
\end{equation}
If $S_i$ is bipartite, then, as a simple bipartite graph of order $2n_i$, it
has at most $n_i^2$ edges. Suppose that $S_i$ is nonbipartite. Then
$n_i\ne1$, and so $n_i\ge2$. Since $k\ge2$ and $n_j\ge1$ for every $j$ ($j\ne i$),
equation~\eqref{eq:sum-ni} gives
\[
 n_i=n-\sum_{\substack{1\le j\le k\\j\ne i}}n_j\le n-(k-1)\le n-1,
 \qquad \text{for }1\le i\le k.
\]
Hence $2\le n_i<n$. By the induction hypothesis,
$|E(S_i)|\le n_i^2+2$.

In the following, we will use $|E(S_i)|$ to derive $|E(G)|$.
Now, we let $\beta_i$ be given as
\begin{equation}
 \beta_i=
 \begin{cases}
  1, & \text{if } S_i \text{ is nonbipartite},\\
  0, & \text{if } S_i \text{ is bipartite},
 \end{cases}\qquad \text{for }1\le i\le k.
 \label{eq:beta}
\end{equation}
With this notation, set
\[
 \delta_i=n_i^2+2\beta_i-|E(S_i)|,
 \qquad \text{for }1\le i\le k.
\]
If $S_i$ is bipartite, then $\delta_i=n_i^2-|E(S_i)|$, while
$|E(S_i)|\le n_i^2$. If $S_i$ is nonbipartite, then
$\delta_i=n_i^2+2-|E(S_i)|$, while
$|E(S_i)|\le n_i^2+2$. Hence
$\delta_i\ge0$.
Next, let
\[
 \rho_i=kp_i-|E[V(Q_i),B]|,
 \qquad \text{for }1\le i\le k.
\]
By the definition of $p_i$, precisely $p_i$ vertices of $Q_i$ have
neighbours in $B$. For each $x\in V(Q_i)$ with $N_G(x)\cap B\ne\varnothing$,
the vertex $x$ forms a pair $(x,b)$ with each vertex $b\in B$. Since $|B|=k$,
there are $kp_i$ such pairs. Since $G$ is simple,
precisely $|E[V(Q_i),B]|$ of these pairs are pairs of adjacent vertices.
Thus, $\rho_i$ is precisely the number of pairs of nonadjacent vertices among
them, and hence $\rho_i\ge0$.
Finally, set
\[
 \tau_i=2n_i-1-p_i,
 \qquad \text{for }1\le i\le k.
\]
Since $|V(Q_i)|=2n_i-1$, there are precisely $\tau_i$ vertices of $Q_i$ that have
no neighbour in $B$. Hence $\tau_i\ge0$.

By the definitions of $\delta_i$, $\rho_i$ and $\tau_i$, we have
$|E(S_i)|=n_i^2+2\beta_i-\delta_i$,
$|E[V(Q_i),B]|=kp_i-\rho_i$ and $p_i=2n_i-1-\tau_i$ for
$1\le i\le k$.
These equalities, together with \eqref{eq:si-edge-count}, imply that
\[
 \begin{aligned}
 |E(Q_i)|+|E[V(Q_i),B]|
 &=|E(S_i)|-p_i+|E[V(Q_i),B]|\\
% &=n_i^2+2\beta_i-\delta_i-p_i+|E[V(Q_i),B]|\\
 &=n_i^2+2\beta_i-\delta_i-p_i+kp_i-\rho_i\\
% &=n_i^2+2\beta_i-\delta_i+(k-1)p_i-\rho_i\\
 &=n_i^2+2\beta_i-\delta_i+(k-1)(2n_i-1-\tau_i)-\rho_i\\
 &=n_i^2+2\beta_i+(k-1)(2n_i-1)\\
 &\quad-\bigl(\delta_i+\rho_i+(k-1)\tau_i\bigr),
 \qquad \text{for }1\le i\le k.
 \end{aligned}
\]
By adding the above identities for $i=1,\ldots,k$ and using
\eqref{eq:barrier-partition}, we obtain
\begin{equation}
 \begin{split}
 |E(G)|={}&\sum_{i=1}^{k}
 \bigl(n_i^2+2\beta_i+(k-1)(2n_i-1)\bigr)\\
 &-\sum_{i=1}^{k}\bigl(\delta_i+\rho_i+(k-1)\tau_i\bigr).
\end{split}
\label{eq:barrier-defect-one}
\end{equation}

Let
\[
 \lambda_i=n_i-1,\qquad \text{for }1\le i\le k.
\]
Since $|V(Q_i)|=2n_i-1$, we have
$\lambda_i=\frac{|V(Q_i)|-1}{2}$. Thus, $\lambda_i$ is a nonnegative integer;
$\lambda_i=0$ precisely when $Q_i$ is a single vertex, whereas $\lambda_i\ge1$ precisely
when $Q_i$ has at least three vertices. For $1\le i<j\le k$, the product
$\lambda_i\lambda_j$ is positive precisely when both $Q_i$ and $Q_j$ have at least three
vertices. Consequently,
$\sum_{1\le i<j\le k}\lambda_i\lambda_j=0$ if and only if at most one component of
$G-B$ has more than one vertex.

By \eqref{eq:sum-ni} and the definition of $\lambda_i$, we have
\[
 \sum_{i=1}^{k}\lambda_i=n-k
\]
and
\[
 \sum_{i=1}^{k}(2n_i-1)=2n-k.
\]
In particular, $n=k+\sum_{i=1}^{k}\lambda_i$. Using this identity and
$n_i=\lambda_i+1$, we have
\[
 \begin{aligned}
 \sum_{i=1}^{k}n_i^2+(k-1)(2n-k)
 &=\sum_{i=1}^{k}(\lambda_i+1)^2+(k-1)(2n-k)\\
 &=\sum_{i=1}^{k}\lambda_i^2+2\sum_{i=1}^{k}\lambda_i+k
   +(k-1)\left(k+2\sum_{i=1}^{k}\lambda_i\right)\\
 &=\sum_{i=1}^{k}\lambda_i^2+2k\sum_{i=1}^{k}\lambda_i+k^2\\
 &=\left(k+\sum_{i=1}^{k}\lambda_i\right)^2
   -2\sum_{1\le i<j\le k}\lambda_i\lambda_j\\
 &=n^2-2\sum_{1\le i<j\le k}\lambda_i\lambda_j.
 \end{aligned}
\]

Let $r=\sum_{i=1}^{k}\beta_i$. By \eqref{eq:beta}, $\beta_i=1$ precisely
when $S_i$ is nonbipartite. Thus, $r$ is the number of nonbipartite graphs
among $S_1,\ldots,S_k$. In particular, $\sum_{i=1}^{k}2\beta_i=2r$.
Using \eqref{eq:barrier-defect-one} and the identities above, we have
\begin{equation}
 \begin{aligned}
 n^2+2-|E(G)|
 &=n^2+2-\sum_{i=1}^{k}
 \bigl(n_i^2+2\beta_i+(k-1)(2n_i-1)\bigr)
 +\sum_{i=1}^{k}\bigl(\delta_i+\rho_i+(k-1)\tau_i\bigr)\\
 &=n^2+2-\left(\sum_{i=1}^{k}n_i^2+(k-1)(2n-k)\right)-2r
 +\sum_{i=1}^{k}\bigl(\delta_i+\rho_i+(k-1)\tau_i\bigr)\\
 &=n^2+2-\left(n^2-2\sum_{1\le i<j\le k}\lambda_i\lambda_j\right)-2r
 +\sum_{i=1}^{k}\bigl(\delta_i+\rho_i+(k-1)\tau_i\bigr)\\
 &=2\left(1-r+\sum_{1\le i<j\le k}\lambda_i\lambda_j\right)
 +\sum_{i=1}^{k}\bigl(\delta_i+\rho_i+(k-1)\tau_i\bigr).
 \end{aligned}
 \label{eq:barrier-defect-two}
\end{equation}

If $\beta_i=1$, then $S_i$ is a nonbipartite graph of even order and hence
has at least four vertices. Since $S_i$ has order $2n_i$, we have $n_i\ge2$
and hence $\lambda_i=n_i-1\ge1=\beta_i$. If $\beta_i=0$, then
$\lambda_i\ge0=\beta_i$. In either case, $\lambda_i\ge\beta_i$ for $1\le i\le k$.
Consequently,
\[
 \sum_{1\le i<j\le k}\lambda_i\lambda_j
 \ge\sum_{1\le i<j\le k}\beta_i\beta_j
 =\frac12\left(\left(\sum_{i=1}^{k}\beta_i\right)^2
 -\sum_{i=1}^{k}\beta_i^2\right)
 =\frac12(r^2-r),
\]
where the last equality follows from the definition of $r$ and the fact that
$\beta_i^2=\beta_i$ for $1\le i\le k$. Suppose that $r\ge2$. It follows that
\[
 1-r+\sum_{1\le i<j\le k}\lambda_i\lambda_j
 \ge 1-r+\frac{r^2-r}{2}
 =\frac{(r-1)(r-2)}2\ge0.
\]
If $r=0$, then
$1-r+\sum_{1\le i<j\le k}\lambda_i\lambda_j=1+
\sum_{1\le i<j\le k}\lambda_i\lambda_j\ge1$. If $r=1$, then
$1-r+\sum_{1\le i<j\le k}\lambda_i\lambda_j=
\sum_{1\le i<j\le k}\lambda_i\lambda_j\ge0$. Thus, in all cases,
$1-r+\sum_{1\le i<j\le k}\lambda_i\lambda_j\ge0$.
Moreover, $\delta_i,\rho_i,\tau_i\ge0$ for $1\le i\le k$, and $k-1\ge1$.
Hence $\delta_i+\rho_i+(k-1)\tau_i$ in
\eqref{eq:barrier-defect-two} is nonnegative. It follows from
\eqref{eq:barrier-defect-two} that $n^2+2-|E(G)|\ge0$, and therefore
$$|E(G)|\le n^2+2.$$ This proves the upper bound in Theorem \ref{thm:main}. In the following, we will consider the graphs acheiving the upper bound.

Assume that $|E(G)|=n^2+2$. By \eqref{eq:barrier-defect-two},
\[
 2\left(1-r+\sum_{1\le i<j\le k}\lambda_i\lambda_j\right)
 +\sum_{i=1}^{k}\bigl(\delta_i+\rho_i+(k-1)\tau_i\bigr)=0.
\]
Both terms on the left-hand side are nonnegative and hence are zero. Thus,
\begin{equation}
 1-r+\sum_{1\le i<j\le k}\lambda_i\lambda_j=0.
 \label{eq:barrier-equality}
\end{equation}
Since $\delta_i,\rho_i,\tau_i\ge0$ and $k-1>0$, it follows that
\begin{equation}
 \delta_i=\rho_i=\tau_i=0,\quad \text{for }1\le i\le k.
 \label{eq:barrier-zero-defects}
\end{equation}

For $1\le i\le k$, the equality $\tau_i=2n_i-1-p_i=0$ gives
$p_i=2n_i-1=|V(Q_i)|$, and
the equality $\rho_i=kp_i-|E[V(Q_i),B]|=0$ then gives
\begin{equation}
 |E[V(Q_i),B]|=kp_i=k|V(Q_i)|.
 \label{eq:barrier-cross-edges}
\end{equation}
Since $G$ is simple, $|B|=k$, and $p_i$ is the number of vertices of $Q_i$
having neighbours in $B$, it follows from \eqref{eq:barrier-cross-edges} that
every vertex of $Q_i$ is adjacent to every vertex of $B$. Thus every vertex of
$V(G)-B$ is adjacent to every vertex of $B$. Therefore, the contraction vertex $c_i$ of
$S_i$ is adjacent to every vertex of $Q_i$. Moreover, $\delta_i=n_i^2+2\beta_i-|E(S_i)|=0$ gives
\begin{equation}
 |E(S_i)|=n_i^2+2\beta_i,\qquad \text{for }1\le i\le k.
 \label{eq:barrier-si-equality}
\end{equation}

If $r=0$, then
\[
 1-r+\sum_{1\le i<j\le k}\lambda_i\lambda_j
 =1+\sum_{1\le i<j\le k}\lambda_i\lambda_j>0,
\]
contrary to \eqref{eq:barrier-equality}. If $r\ge3$,
then
\[
 1-r+\sum_{1\le i<j\le k}\lambda_i\lambda_j
 \ge 1-r+\frac{r^2-r}{2}=\frac{(r-1)(r-2)}2>0,
\]
which contradicts \eqref{eq:barrier-equality}.
Thus, $1\le r\le2$.

Suppose first that $r=1$. Then precisely one of $S_1,\ldots,S_k$ is
nonbipartite. Adjust notation so that $S_1$ is nonbipartite. Then
$\beta_1=1$ by \eqref{eq:beta}. Since $S_1$ is nonbipartite and has order
$2n_1$, we have $n_1\ge2$ and hence $\lambda_1=n_1-1\ge1$, whereas the equality
\eqref{eq:barrier-equality} gives
\[
 \sum_{1\le i<j\le k}\lambda_i\lambda_j=0.
\]
Since every $\lambda_i$ is nonnegative, $\lambda_1\lambda_j=0$ for $2\le j\le k$. Hence
$\lambda_2=\cdots=\lambda_k=0$, and $Q_2,\ldots,Q_k$ are trivial components of $G-B$.
Since $S_1$ is a simple solid matching covered graph of order $2n_1$ and
$2\le n_1<n$, the induction hypothesis applies to $S_1$. By
\eqref{eq:barrier-si-equality}, $|E(S_1)|=n_1^2+2$.
By the induction hypothesis and Proposition~\ref{prop:small-equality}, we have either
$n_1=2$ and
$S_1\cong A_2=K_4$; $n_1=3$ and $S_1\cong A_3$ or $S_1\cong D_3$; or
$n_1\ge4$ and $S_1\cong A_{n_1}$, $S_1\cong B_{n_1}$ or
$S_1\cong D_{n_1}$.

Now we will determine the structure of $G$ when $|E(G)|=n^2+2$.
The contraction vertex of $S_1$ is adjacent to every other vertex and hence
has degree $2n_1-1$. By definition, the maximum degree
$\Delta(A_{n_1})=n_1+1$, $\Delta(B_{n_1})=n_1+2$ and
$\Delta(D_{n_1})=n_1+2$.
If $S_1\cong A_{n_1}$, then $2n_1-1\le n_1+1$. Since $n_1\ge2$, this gives
$n_1=2$.
If $S_1\cong B_{n_1}$, then $2n_1-1\le n_1+2$, so $n_1\le3$, which is contrary to $n_1\ge4$.
If $S_1\cong D_{n_1}$, then $2n_1-1\le n_1+2$. Since $n_1\ge3$, this gives
$n_1=3$. Therefore, $S_1\cong A_2$ or $S_1\cong D_3$. Deleting the
contraction vertex from $S_1$ leaves $Q_1$. Hence $Q_1\cong K_3$ or
$Q_1\cong F$, respectively.

If $Q_1\cong K_3$, then $3=|V(Q_1)|=2n_1-1$, and hence $n_1=2$. Since
$Q_2,\ldots,Q_k$ are trivial, $n_i=1$ for $2\le i\le k$. Thus,
\eqref{eq:sum-ni} gives $n=\sum_{i=1}^{k}n_i=2+(k-1)=k+1$, and hence
$k=n-1$. Therefore, $G-B$ has $Q_1$ as its only nontrivial component and
$n-2$ trivial components.  Since $k\ge2$, $n\ge3$. As every vertex of $G-B$
is adjacent to every vertex of $B$, we have
\[
 G\cong\comp{K}_{n-1}\vee\bigl(K_3\cup(n-2)K_1\bigr)=A_n.
\]
If $Q_1\cong F$, then $5=|V(Q_1)|=2n_1-1$, so $n_1=3$. As
$Q_2,\ldots,Q_k$ are trivial components, $n_i=1$ for $2\le i\le k$.
By \eqref{eq:sum-ni}, we obtain
$n=\sum_{i=1}^{k}n_i=3+(k-1)=k+2$, and consequently $k=n-2$.
Hence $Q_1$ is the unique nontrivial component of $G-B$, while the other
$n-3$ components are trivial. Since $k\ge2$, $n\ge4$. By the definition of
$D_n$,
\[
 G\cong\comp{K}_{n-2}\vee\bigl(F\cup(n-3)K_1\bigr)=D_n.
\]

Suppose next that $r=2$. Then precisely two of $S_1,\ldots,S_k$ are
nonbipartite. Adjust notation so that $S_1$ and $S_2$ are nonbipartite. Then
$\beta_1=\beta_2=1$ by \eqref{eq:beta}. Since the simple graphs $S_1$ and $S_2$ are
nonbipartite and have orders $2n_1$ and $2n_2$, respectively, we have
$n_1\ge2$ and $n_2\ge2$. Hence $\lambda_1=n_1-1\ge1$ and $\lambda_2=n_2-1\ge1$.
Equation~\eqref{eq:barrier-equality} gives
\[
 \sum_{1\le i<j\le k}\lambda_i\lambda_j=1.
\]
Since $\lambda_1\lambda_2\ge1$ and every term in the sum is nonnegative, the equality
above gives $\lambda_1\lambda_2=1$ and $\lambda_1\lambda_j=0$ for $3\le j\le k$. Since $\lambda_1$ and
$\lambda_2$ are positive integers, $\lambda_1=\lambda_2=1$, and hence $\lambda_j=0$ for
$3\le j\le k$. Since $\lambda_i=n_i-1$, we have $n_1=n_2=2$ and $n_j=1$ for
$3\le j\le k$. Hence every component of $G-B$ other than $Q_1$ and $Q_2$
is trivial. By \eqref{eq:barrier-si-equality},
$|E(S_1)|=n_1^2+2\beta_1=6$ and
$|E(S_2)|=n_2^2+2\beta_2=6$. Since $S_1$ and $S_2$ are simple graphs of
order four, $S_1\cong S_2\cong K_4$. Deleting the contraction vertex from
$S_i$ leaves $Q_i$, for $i=1,2$. Hence $Q_1\cong Q_2\cong K_3$.

By \eqref{eq:sum-ni},
$n=\sum_{i=1}^{k}n_i=2+2+(k-2)=k+2$, and hence $k=n-2$. Since $k\ge2$, we
have $n\ge4$, and there are $n-4$ trivial components of $G-B$. We have shown
that every vertex of $G-B$ is adjacent to every vertex of $B$. It
follows that
\[
 G\cong\comp{K}_{n-2}\vee\bigl(2K_3\cup(n-4)K_1\bigr)=B_n.
\]
This completes Case~1.

\smallskip
\noindent\textbf{\upshape Case 2:}
The ELP cut $C$ is a 2-separation cut associated with a 2-separation
$\{u,v\}$ of $G$. Let
$X$ be the union of the vertex sets of a nonempty proper subset of the components
of $G-u-v$, and let $Y$ be the union of the vertex sets of the remaining
components. Let
\[
 C=\partial(X\cup\{u\})=\partial(Y\cup\{v\}).
\]
Since $\{u,v\}$ is a 2-separation of $G$, every component of $G-u-v$ has even order and no edge of $G$ joins a vertex of $X$ to a vertex of $Y$.
Thus, $X$ and $Y$ have even cardinality.
Set
\[
 |X|=2(s-1)
\]
and
\[
 |Y|=2(t-1).
\]
Since $X$ and $Y$ are nonempty and have even cardinality, $s,t\ge2$.
Counting the vertices of $G$ gives
\begin{equation}
 s+t=n+1. \label{eq:st}
\end{equation}

Let $H_X$ be a graph obtained from $G$ by contracting $Y\cup\{v\}$ to a single
vertex, and let $H_Y$ be a graph obtained from $G$ by contracting
$X\cup\{u\}$ to a single vertex. Thus $H_X$ and $H_Y$ are the two
$C$-contractions of $G$. The cut $C$ is
a 2-separation cut and hence is tight by Lemma~\ref{lem:elp-cuts-tight}. Since
every tight cut is a separating cut, $H_X$ and $H_Y$ are matching covered. Since $G$
is solid, Theorem~\ref{thm:solid-tight-contractions} implies that $H_X$ and $H_Y$
are solid.

Let $S_X=\operatorname{Sim}(H_X)$ and $S_Y=\operatorname{Sim}(H_Y)$ be the
underlying simple graphs. By Lemma~\ref{lem:simplification}, $S_X$ and $S_Y$ are
simple solid matching covered graphs. Since $|X|=2(s-1)$, the graphs $H_X$ and
$S_X$ have order $2s$. Similarly, $H_Y$ and $S_Y$ have order
$2t$.

By Lemma~\ref{lem:two-connected}, $G$ is 2-connected. Thus both $u$ and $v$ have
a neighbour in every component of $G-u-v$; otherwise, one of $u$ and $v$ would
be a cut vertex of $G$. Hence each of $u$ and $v$ has a neighbour in both $X$
and $Y$.

Let $y$ denote the contraction vertex of $H_X$. Since $u$ has a neighbour in $Y$,
the vertices $u$ and $y$ are adjacent in $H_X$. Since $G$ is simple and no edge
of $G$ joins $X$ and $Y$, the graph $S_X$ is obtained from $H_X$ by replacing
the parallel edges joining $u$ and $y$ by the single edge $uy$. Therefore,
\[
 |E(S_X)|=|E(G[X])|+|E[X,\{u,v\}]|+1.
\]
Similarly, let $x$ denote the contraction vertex of $H_Y$. Since $v$ has a
neighbour in $X$, the vertices $v$ and $x$ are adjacent in $H_Y$. Since $G$ is
simple and no edge of $G$ joins $X$ and $Y$, the graph $S_Y$ is obtained from
$H_Y$ by replacing the parallel edges joining $v$ and $x$ by the single edge
$vx$. Therefore,
\[
 |E(S_Y)|=|E(G[Y])|+|E[Y,\{u,v\}]|+1.
\]

Set
\[
 \gamma=
 \begin{cases}
  1, & \text{if } uv\in E(G)\\
  0, & \text{if } uv\notin E(G).
 \end{cases}
\]
Since no edge of $G$ joins $X$ and $Y$, the four edge sets $E(G[X])$, $E(G[Y])$,
$E[X,\{u,v\}]$ and $E[Y,\{u,v\}]$, together with $\{uv\}$ when
$\gamma=1$, constitute a partition of $E(G)$. Therefore,
\begin{equation}
 \begin{aligned}
 |E(G)|
 &=\bigl(|E(G[X])|+|E[X,\{u,v\}]|\bigr)\\
 &\quad+\bigl(|E(G[Y])|+|E[Y,\{u,v\}]|\bigr)+\gamma\\
 &=\bigl(|E(S_X)|-1\bigr)+\bigl(|E(S_Y)|-1\bigr)+\gamma\\
 &=|E(S_X)|+|E(S_Y)|-2+\gamma.
 \end{aligned}
 \label{eq:two-separation}
\end{equation}

Since $s,t\ge2$ and $s+t=n+1$, we have $2\le s,t<n$. Applying the induction
hypothesis to $S_X$ and $S_Y$ and using \eqref{eq:two-separation}, we obtain
\begin{equation}
\begin{split}
 |E(G)|&=|E(S_X)|+|E(S_Y)|-2+\gamma\\
       &\le(s^2+2)+(t^2+2)-2+\gamma\\
       &=s^2+t^2+2+\gamma\le s^2+t^2+3.
 \end{split}
 \label{eq:two-separation-bound}
\end{equation}
Using \eqref{eq:two-separation-bound}, \eqref{eq:st} and $s,t\ge2$, we obtain
\begin{equation}
 \begin{aligned}
 n^2+2-|E(G)|
 &\ge n^2+2-(s^2+t^2+3)\\
 &=(s+t-1)^2-s^2-t^2-1\\
 &=2st-2s-2t\\
 &=2(st-s-t)\\
 &=2\bigl((s-1)(t-1)-1\bigr)\ge0.
 \end{aligned}
 \label{eq:two-separation-gap}
\end{equation}
It follows that $|E(G)|\le n^2+2$.

Assume that $|E(G)|=n^2+2$. By \eqref{eq:two-separation-bound} and
\eqref{eq:two-separation-gap},
\[
 n^2+2=|E(G)|\le s^2+t^2+3\le n^2+2.
\]
Thus, equality holds throughout. Equality in \eqref{eq:two-separation-gap}
gives $(s-1)(t-1)=1$, and hence $s=t=2$. Equality in the last inequality of
\eqref{eq:two-separation-bound} gives $\gamma=1$, and hence $uv\in E(G)$.
Since $|E(S_X)|\le s^2+2$ and $|E(S_Y)|\le t^2+2$, equality in the first
inequality of \eqref{eq:two-separation-bound} gives
\[
 |E(S_X)|=s^2+2
\]
and
\[
 |E(S_Y)|=t^2+2.
\]
Since $s=t=2$, by \eqref{eq:st}, we have $n=3$. Thus $|X|=|Y|=2$. Moreover,
$S_X$ and $S_Y$ are simple graphs of order four with six edges, so
$S_X\cong S_Y\cong K_4$.
Since no edge of $G$ joins $X$ and $Y$, the isomorphisms
$S_X\cong S_Y\cong K_4$ imply that
$G[X]\cong G[Y]\cong K_2$ and every vertex of $X\cup Y$ is adjacent to both
$u$ and $v$. As $uv\in E(G)$, we conclude that
\[
 G\cong K_2\vee(2K_2)=D_3.
\]
Conversely, by Lemma~\ref{lem:families}, every graph listed in the theorem has
exactly $n^2+2$ edges.
This completes the proof of Theorem~\ref{thm:main}.
\end{proof}

\end{document}